\documentclass[10pt,a4paper,oneside,final]{article}
\usepackage{amsmath}
\usepackage{amssymb}
\usepackage{amsthm}
\usepackage{esint}
\usepackage{color}
\usepackage{fullpage}
\usepackage[pdftex]{graphicx}
\usepackage{subfig}
\usepackage{morefloats}
\usepackage{pgfplots}
\usepackage{rotating}

\usepackage{tikz}
\newcommand\circlearound[1]{\tikz[baseline]\node[draw,shape=circle,scale=0.6,anchor=base] {#1} ;}

\newcommand{\R}{{\mathbb{R}}}

\newcommand{\N}{{\mathbb{N}}}
\newcommand{\de}{{\mathrm{d}}}
\DeclareMathOperator{\dv}{div}

\newcommand{\lebesgue}{\mathcal{L}}

\newcommand{\hd}{\mathcal{H}}
\newcommand{\hdone}{\mathcal{H}^1}

\newcommand{\rca}{\mathrm{rca}}
\newcommand{\fbm}{{\mathrm{fbm}}}

\newcommand{\restr}{{\mbox{\LARGE$\llcorner$}}}
\newcommand{\spt}{{\mathrm{spt}}}

\newcommand{\weakstarto}{\stackrel{*}{\rightharpoonup}}
\newcommand{\Wd}[1]{\mathrm{W}_{#1}}
\newcommand{\Wdone}{\Wd{1}}

\newcommand{\BV}{{\mathrm{BV}}}
\newcommand{\SBV}{{\mathrm{SBV}}}
\newcommand{\cont}{{\mathcal{C}}}

\newcommand{\smooth}{\mathcal{C}_{c}^{\infty}}
\newcommand{\flux}{{\mathcal{F}}}

\newcommand{\J}{{\mathcal{J}}}
\newcommand{\brTptEn}{{\mathcal{M}}}
\newcommand{\urbPlEn}{{\mathcal{E}}}
\newcommand{\XiaEn}{M}

\newcommand{\urbPlXia}{E}

\newcommand{\brTptImg}{\tilde M}
\newcommand{\urbPlImg}{\tilde E}

\newcommand{\Kcal}{\mathcal{K}}

\newcommand{\keywords}[1]{\noindent\textbf{Keywords:}\enspace#1}
\newcommand{\subjclass}[1]{\bigskip\noindent\emph{2010 MSC:}\enspace#1}

\numberwithin{equation}{subsection}

\theoremstyle{plain}
\newtheorem{theorem}{Theorem}[subsection]
\newtheorem{lemma}[theorem]{Lemma}

\theoremstyle{definition}
\newtheorem{definition}[theorem]{Definition}

\theoremstyle{remark}
\newtheorem{remark}[theorem]{Remark}

\newcommand{\notinclude}[1]{}

\begin{document}

\title{Optimal micropatterns in 2D transport networks and their relation to image inpainting}
\author{Alessio Brancolini\footnote{Institute for Numerical and Applied Mathematics, University of M\"unster, Einsteinstra\ss{}e 62, D-48149 M\"unster, Germany}\ \footnote{Email address: \texttt{alessio.brancolini@uni-muenster.de} (corresponding author)} \and Carolin Rossmanith\footnotemark[1]\ \footnote{Email address: \texttt{carolin.rossmanith@uni-muenster.de}} \and Benedikt Wirth\footnotemark[1]\ \footnote{Email address: \texttt{benedikt.wirth@uni-muenster.de}}}
\date{}
\maketitle

\begin{abstract}
We consider two different variational models of transport networks, the so-called branched transport problem and the urban planning problem.
Based on a novel relation to Mumford--Shah image inpainting and techniques developed in that field,
we show for a two-dimensional situation that both highly non-convex network optimization tasks can be transformed into a convex variational problem,
which may be very useful from analytical and numerical perspectives.

As applications of the convex formulation, we use it to perform numerical simulations (to our knowledge this is the first numerical treatment of urban planning),
and we prove the lower bound of an energy scaling law
which helps better understand optimal networks and their minimal energies.

\bigskip\keywords{micropatterns, image inpainting, energy scaling laws, optimal transport, optimal networks, branched transport, irrigation, urban planning, Wasserstein distance, convex lifting, convex optimization}

\end{abstract}

\section{Introduction}
The optimization of transport networks, in particular the so-called urban planning \cite{BuPrSoSt09} or the branched transport problem \cite{BeCaMo09}, is a nontrivial task.
The corresponding cost functionals are highly non-convex, and their minimizers, the optimal networks, often exhibit a complicated, strongly ramified structure.

\subsection{Convexification, numerical optimization, and energy scaling laws}
A better understanding can either be achieved by numerical or by analytical means.

Via numerical optimization one may for instance identify and explore the optimal network structures for different parameters or geometric settings.
This then provides an intuition of the model behaviour.
While for urban planning we are not aware of any numerical treatment, optimal networks in branched transport have already been computed using a phase field approximation \cite{OuSa11}.
However, due to the complicated energy landscape it is highly nontrivial not to get stuck in local minima, and in general there is no guarantee to achieve the global optimum (which is known to exist).

A classical alternative, analytical way to achieve a better understanding of the network models is by proving how the optimal cost scales in the problem parameters.
To this end one typically restricts to a specific, simple situation and constructs a cost efficient network for that situation.
If one can prove that, up to a constant factor, no network can attain a lower cost, then this provides information about the minimum achievable cost and the structure of near-optimal networks.
In particular it can be learned how the minimum cost scales in the problem parameters.

Such energy scaling laws have already successfully been used to explain a number of complicated patterns seen in physical experiments
such as martensite--austenite transformations \cite{KM92,KoMu94,KnKoOt13,ChCo14,BeGo14},
micromagnetics \cite{ChKoOt99}, intermediate states in type-I superconductors \cite{ChKoOt04,ChCoKo08},
membrane folding and blistering \cite{BeKo14,BeKo15}, or epitaxial growth \cite{GoZw14}.
They were also employed to find the optimal fine-scale structure of composite elastic materials \cite{KoWi14,KoWi14b}
as well as of branched transport and urban planning networks \cite{BrWi15-micropatterns}.

In this article we perform both numerical optimization as well as energy scaling analysis.
In particular, we compute optimal urban planning networks numerically for the first time,
and we redo the energy scaling analysis of branched transport and urban planning networks.
To these ends we follow a novel approach which is interesting on its own.
For a two-dimensional situation we show that the problem can actually be convexified.
This is achieved by exploiting a novel relation to image inpainting and techniques developed in that field.

We believe that the convex formulation is in many situations tight, that is, equivalent to the original network optimization,
however, we only prove that it represents a lower bound.
Exploring when it is also an upper bound will be more complicated and certainly requires quite technical constructions (similar to those involved in the regularity analysis of the Mumford--Shah functional; we refer to Remark \ref{rem:upper_bound} and the final discussion in Section \ref{sec:discussion}).
We will only briefly discuss the associated analytical problems and instead showcase numerical results that support the tightness of the convexification except for a few special cases.

Convex formulations have the great advantage that the local behaviour of the cost or energy provides global information of the full energy landscape,
which can be used in many ways (for instance numerically to find the global energy minimizer or analytically to find lower energy bounds by convex duality).
As applications we perform numerical network optimization, and we reprove the tight lower bound on the minimum achievable network cost, which is the essential ingredient of the energy scaling law.

\subsection{Branched transport and urban planning}
Branched transport has been introduced in different formulations by Maddalena, Morel, and Solimini \cite{Maddalena-Morel-Solimini-Irrigation-Patterns} and by Xia \cite{Xia-Optimal-Paths} and then studied by various authors; a comprehensive treatment is found in the monograph \cite{BeCaMo09}, while regularity issues are considered in \cite{Morel-Santambrogio-Regularity}, \cite{Santambrogio-Landscape}, \cite{Brancolini-Solimini-Hoelder}, and \cite{Brancolini-Solimini-Fractal}.
It has been used to describe hierarchically branched networks such as the blood vessel system or the water supply network in plants.
The efficiency of a network is judged by a functional, which measures the cost associated with the transport of mass from a given source distribution to a sink distribution.
This functional has the feature that the transport cost per transport distance is not proportional to the transported mass, but increases sublinearly in the mass.
In other words, the more mass is transported together, the lower is cost per single particle, which models an increase in transport efficiency.
This feature leads to ramified networks, since it becomes more efficient to first collect all the mass via a hierarchically branched network and then transport it all together in bulk.
The strength of this effect and thus the degree of ramification is governed by a model parameter $\varepsilon>0$.
For small $\varepsilon$, the effect is only weak, which leads to highly complicated structures with many pipes becoming optimal.

The urban planning model was introduced in \cite{Brancolini-Buttazzo} and is motivated by public transport networks.
Here, the given source represents the distribution of homes in an urban area, and the sink represents the distribution of workplaces.
Every commuter can travel on the network, paying a fixed cost of $1$ per travel distance, or outside the network by own means, paying a slightly larger cost $a$.
In contrast to branched transport, the resulting total transport cost is proportional to the transported mass.
In addition, the urban planning model has maintenance costs which are proportional to the total network length with proportionality factor $\varepsilon$.
Those maintenance costs again lead to the preference of hierarchically branched networks.

\subsection{The considered setting}
To better understand the model behaviour, an energy scaling law for the network costs has been derived in \cite{BrWi15-micropatterns} (and will be reproved in this work, see Theorem\,\ref{thm:scalingUrbPlan}) for the following simple problem geometry (Figure\,\ref{fig:setting} left).
In two-dimensional Euclidean space one considers as source and sink distribution the following measures concentrated on two lines,
\begin{equation*}
\mu_0=\lebesgue^1\restr[0,\ell]\times\{0\}\,,
\qquad
\mu_1=\lebesgue^1\restr[0,\ell]\times\{1\}\,,
\end{equation*}
where $\ell>0$, $\lebesgue^1$ denotes the one-dimensional Lebesgue measure and $\restr$ indicates its restriction onto a subset of $\R^2$.
Note that the case of measures with a larger or smaller distance between their supports or with a different common total mass can easily be reduced to the above situation
and that \cite{BrWi15-micropatterns} actually also considers the analogous problem in more than two dimensions.
One seeks the optimal network (with respect to the urban planning or branched transport costs) to transport the mass from $\mu_0$ to $\mu_1$.
This situation may for instance be viewed as a strong simplification of the water transport within a single plant or a forest:
$\mu_0$ represents the groundwater reservoir and $\mu_1$ the water consumption in the leaf canopy (Figure\,\ref{fig:setting} right).

\begin{figure}
\setlength{\unitlength}{\linewidth}%
\begin{picture}(1,.5)
\put (.03,.07) { \includegraphics[width=0.64\textwidth]{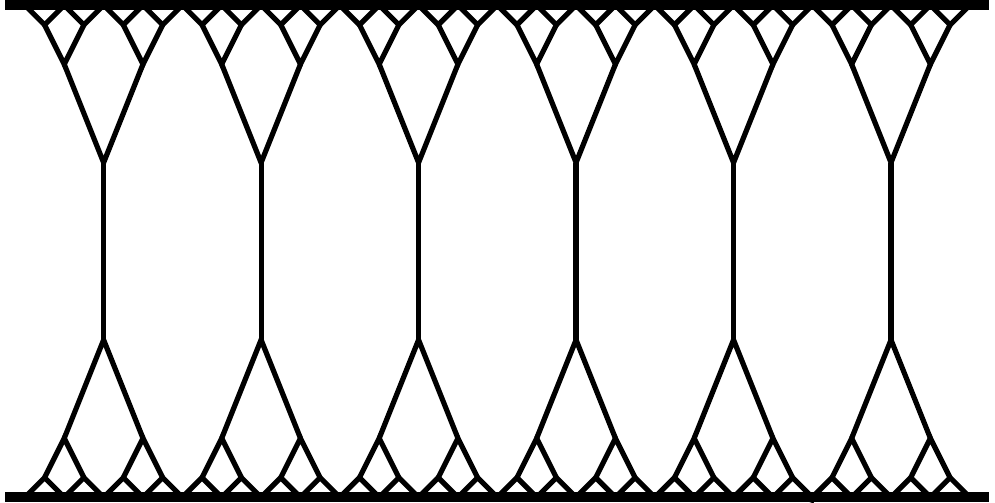} }
\put (.05,.03){\vector(1,0){.625}}
\put(.05,.03){\vector(-1,0){.01}}
\put(.02,.07){\vector(0,1){.32}}
\put(.02,.08){\vector(0,-1){.01}}
\put(.35,.05){$\mu_0$}
\put(.35,.402){$\mu_1$}
\put(.35,0){$\ell$}
\put(0,.22){$1$}
\multiput(.265,.11)(.022,0){4}{\line(1,0){.016}}
\multiput(.265,.175)(.022,0){4}{\line(1,0){.016}}
\multiput(.265,.11)(0,.025){3}{\line(0,1){.015}}
\multiput(.346,.11)(0,.025){3}{\line(0,1){.015}}
\put(.74,.1){\includegraphics[height=.26\unitlength,angle=-90,origin=c]{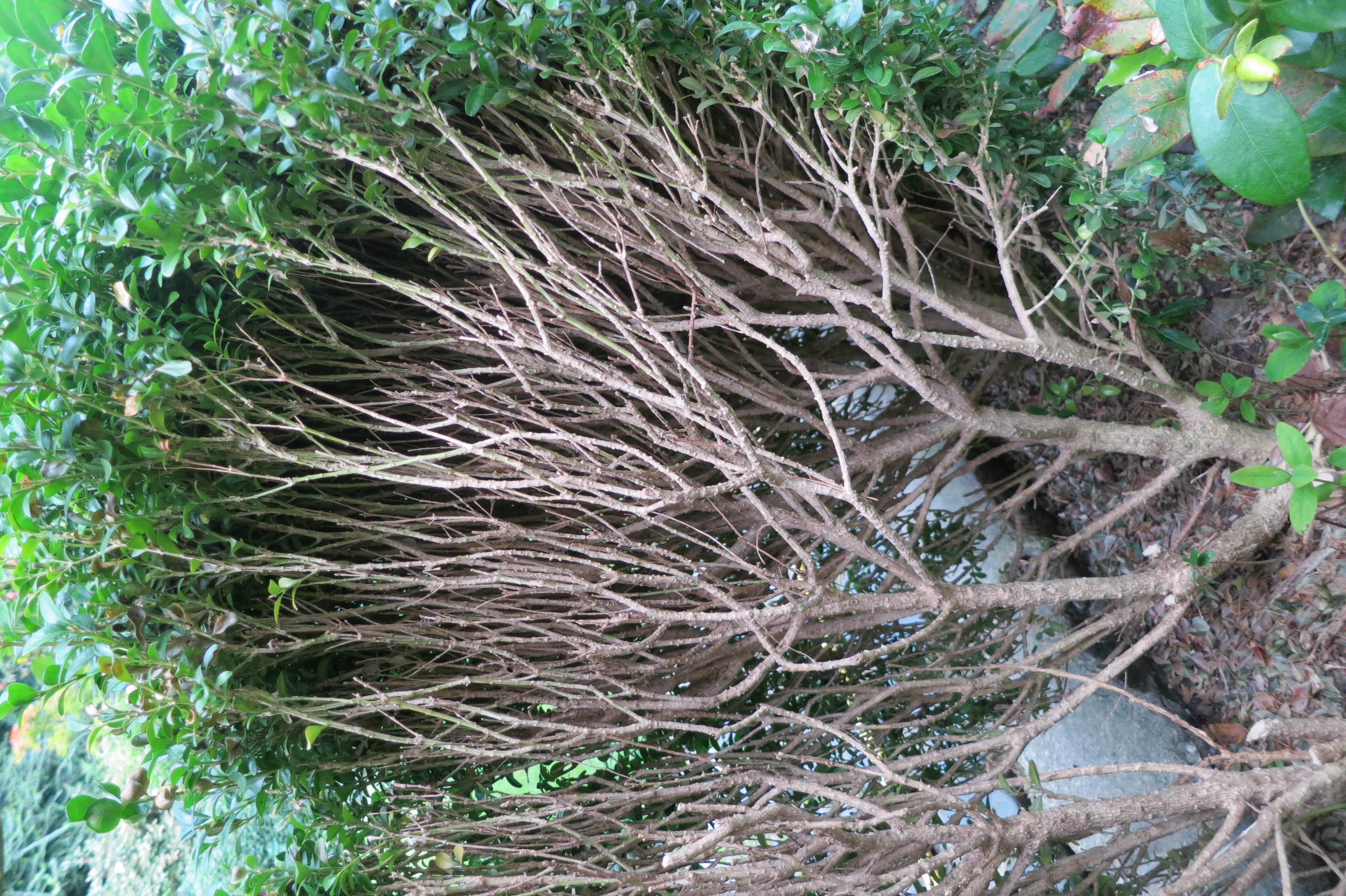}}
\end{picture}
\caption{Left: Sketch of the considered setting, two measures $\mu_0$ and $\mu_1$ supported on lines at distance $1$, as well as an exemplary transport network in between composed of elementary cells (dashed line).
Right: Photograph of a box tree exhibiting a branched network structure and a comparatively thin leaf canopy that may almost be thought of as a lower-dimensional manifold.}
\label{fig:setting}
\end{figure}
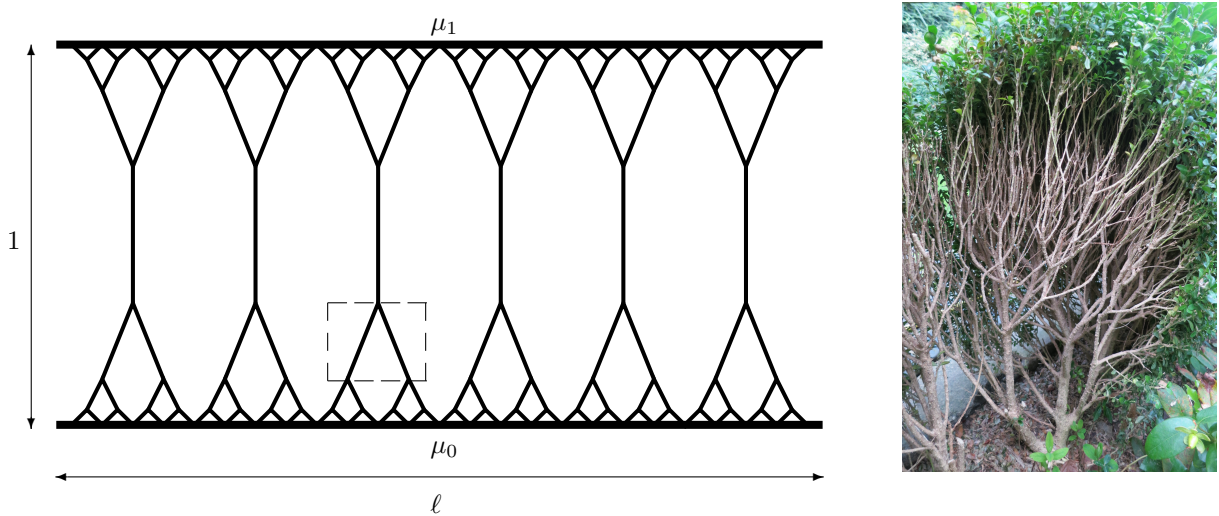

The truly optimal transport network is too hard to find, but the previously described energy scaling analysis allows to examine how near-optimal network patterns look like.
A qualitative sketch is given in Figure\,\ref{fig:setting} left:
such a network is composed of multiple hierarchical levels.
Each level consists of an array of V-shaped elementary cells,
with heights and aspect ratios depending on the level and on the problem parameters $a$ and $\varepsilon$, the details are given in \cite{BrWi15-micropatterns}.

For numerical network optimization we will also consider different geometric settings such as sources and sinks distributed on a circle (see Figure\,\ref{fig:ResultsCircle2}).

\subsection{The connection to the Mumford--Shah problem}
The Mumford--Shah problem is a variational formulation of an image segmentation or denoising task, originally introduced in \cite{MuSh89}.
The corresponding energy functional is designed to smooth out a given noisy image while preserving the edges or discontinuities visible inside the image.
The work was very influential, and many techniques have been devised to deal with the problem numerically and analytically.
Among these techniques there are so-called functional lifting methods that find a convex formulation of the highly complicated and non-convex Mumford--Shah segmentation
at the expense of introducing an additional dimension to the problem \cite{PCBC-2009}.

The essential observation in the context of network optimization now is that a two-dimensional, divergence-free mass flux (the flux running through the network) becomes a gradient field after rotation by $\frac\pi2$.
This gradient field can be interpreted as the gradient of an image, which suddenly allows to reformulate network optimization as a Mumford--Shah-type image processing problem and thus to use the functional lifting techniques.

The outline of the article is as follows.
We will introduce urban planning and branched transport in Section\,\ref{sec:FluxFormulation} as well as the energy scaling result, which is the main analytical application of our new convex formulation.
The relation between network optimization and image processing is derived in Section\,\ref{sec:MumfordShah}, as is the convex reformulation.
Section\,\ref{sec:MumfordShah} further contains the applications, the lower bound of the energy scaling as well as a numerical approach based on the reformulation.
We close with a discussion in Section\,\ref{sec:discussion}.

\section{Transport networks and their energy scaling}\label{sec:FluxFormulation}

We consider two different models for transport networks, so-called branched transport and urban planning.
The formulation in which we state both models is originally due to Xia \cite{Xia-Optimal-Paths};
the fact that also urban planning can be formulated this way has been shown in \cite{BrWi15-equivalent}.

In the following, $\hd^n$ and $\lebesgue^n$ denote the $n$-dimensional Hausdorff and Lebesgue measure,
$\fbm(\Omega)$ is the set of non-negative finite Borel measures on $\Omega\subset\R^n$,
and $\rca(\Omega;\R^n)$ denotes the set of vector-valued regular countably additive measures (Radon measures).
Both $\fbm(\Omega)$ and $\rca(\Omega;\R^n)$ are equipped with the notion of weak-$*$ convergence, denoted by $\weakstarto$.
Finally, the Dirac delta distribution at a point $x\in\R^n$ will be denoted by $\delta_x$.

\subsection{Flux-based formulations of urban planning and branched transport}
\begin{definition}[Discrete mass flux]\label{def:discrete_mass_flux}
A \emph{transport path} is a weighted directed graph $G$ in $\R^n$ with vertices $V(G)$, straight edges $E(G)$, and weight function $w : E(G) \to [0,\infty)$.

Denote the initial and final point of an edge $e\in E(G)$ by $e^+$ and $e^-$ so that the edge has direction $\hat e=\frac{e^--e^+}{|e^--e^+|}$.
The edge can be identified with the vector measure $\mu_e = (\hdone\restr e)\, \hat e$.
The \emph{discrete mass flux} corresponding to a transport path $G$ is the vector measure
\begin{equation*}\label{eqn:graphFlux}
 \flux_G = \sum_{e \in E(G)} w(e)\mu_e\,.
\end{equation*}

Let $\mu_+ = \sum_{i = 1}^k a_i\delta_{x_i}$, $\mu_- = \sum_{i = 1}^l b_i\delta_{y_i}$ be discrete finite non-negative measures with $a_i, b_i > 0$, $x_i, y_i \in \R^n$ and equal mass $\|\mu_+\|_\fbm=\|\mu_-\|_\fbm$.
$G$ is called a \emph{transport path between $\mu_+$ and $\mu_-$}, or equivalently $\flux_G$ is called a \emph{discrete mass flux between $\mu_+$ and $\mu_-$}, if $\dv\flux_G=\mu_+-\mu_-$ in the distributional sense.
\end{definition}

The measure $\mu_+$ and $\mu_-$ can be thought of as a mass source and sink, respectively,
and the transport path $G$ or equivalently the discrete mass flux $\flux_G$ describe how the mass is transported from the source to the sink.
Here, the edge weight $w(e)$ indicates how much mass is flowing along edge $e$,
and the divergence condition $\dv\flux_G=\mu_+-\mu_-$ ensures that no mass gets lost on the way from $\mu_+$ to $\mu_-$.

\begin{definition}[Urban planning and branched transport cost function]\label{def:mass_flux_cost_function_branched_transport_discrete_case}
Let $a>1$, $\varepsilon>0$.
The \emph{urban planning cost} of a transport path $G$ or discrete mass flux $\flux_G$ is given by
\begin{equation*}
\urbPlXia^{\varepsilon,a}(\flux_G)
=\urbPlXia^{\varepsilon,a}(G)
=\sum_{e \in E(G)} \min(aw(e),w(e)+\varepsilon)\, l(e)\,,
\end{equation*}
where $l(e)$ denotes the length of edge $e$.
The \emph{branched transport cost} for $\varepsilon\in(0,1)$ is given by
\begin{equation*}
\XiaEn^{\varepsilon}(\flux_G)
=\XiaEn^{\varepsilon}(G)
=\sum_{e \in E(G)} w(e)^{1-\varepsilon}\, l(e)\,.
\end{equation*}
\end{definition}

The functionals $\urbPlXia^{\varepsilon,a}$ and $\XiaEn^\varepsilon$ model the cost associated with the mass transport encoded by $G$.
The cost contribution of each edge is proportional to the transport distance $l(e)$.
The transport cost per particle and distance is $\min(a,1+\frac\varepsilon{w(e)})$ for urban planning and $w(e)^{-\varepsilon}$ for branched transport.
Obviously, the cost per particle decreases the larger the total flux $w(e)$ through the edge, that is, the more particles travel together.
These economies of scale lead to the fact that branched networks are more cost-efficient than non-branched ones.
Indeed, it pays off to first collect mass together, then transport it inexpensively in bulk, and only later again divert the single particles to their final destinations.
The parameters $\varepsilon$ and $a$ encode the strength of these economies of scale.

As already mentioned before, in urban planning the cost also has an alternative interpretation \cite{BrWi15-micropatterns}:
it can be viewed as the sum of the maintenance costs for the transport network (given by $\varepsilon$ times the total network length)
and the transport costs, which split up into a contribution from mass transport via the network (at cost 1 per particle and distance)
as well as a contribution from mass transport outside the network (at a slightly larger cost $a$ per particle and distance).

\begin{definition}[Continuous mass flux and cost function]
Let $\mu_+,\mu_- \in \fbm(\R^n)$ of equal mass.
A vector measure $\flux\in\rca(\R^n;\R^n)$ is a \emph{mass flux between $\mu_+$ and $\mu_-$},
if there exist sequences of discrete measures $\mu_+^k$, $\mu_-^k$ with $\mu_+^k \weakstarto \mu_+$, $\mu_-^k \weakstarto \mu_-$ as $k\to\infty$,
and a sequence of discrete mass fluxes $\flux_{G_k}$ between $\mu_+^k$ and $\mu_-^k$ with $\flux_{G_k} \weakstarto \flux$.
Note that $\dv\flux=\mu_+-\mu_-$ follows by continuity with respect to weak-$*$ convergence.

A sequence $(\mu_+^k,\mu_-^k,\flux_{G_k})$ satisfying the previous properties is called \emph{approximating graph sequence}, and we write $(\mu_+^k,\mu_-^k,\flux_{G_k}) \weakstarto (\mu_+,\mu_-,\flux)$.

If $\flux$ is a mass flux between $\mu_+$ and $\mu_-$, the \emph{urban planning} and \emph{branched transport cost} are respectively defined as
\begin{align*}\label{eq:functional_XiaEn}
 \urbPlXia^{\varepsilon,a}(\flux) &= \inf\left\{\liminf_{k \to \infty} \urbPlXia^{\varepsilon,a}(G_k) \ : \ (\mu_+^k,\mu_-^k,\flux_{G_k}) \weakstarto (\mu_+,\mu_-,\flux)\right\}\,,\\
 \XiaEn^\varepsilon(\flux) &= \inf\left\{\liminf_{k \to \infty} \XiaEn^\varepsilon(G_k) \ : \ (\mu_+^k,\mu_-^k,\flux_{G_k}) \weakstarto (\mu_+,\mu_-,\flux)\right\}\,.
\end{align*}
\end{definition}

As before, $\flux\in\rca(\R^n;\R^n)$ describes the mass transport from a source $\mu_+$ to a sink $\mu_-$.
One is interested in the optimally cost-efficient mass transport, that is, for fixed $\mu_+$ and $\mu_-$ one seeks the mass flux between $\mu_+$ and $\mu_-$ with least cost $\urbPlXia^{\varepsilon,a}$ or $\XiaEn^{\varepsilon}$.

\begin{remark}[Upper bound on urban planning]\label{rem:urban_planning_bounded_by_a_times_wasserstein_one}
For later reference, note that in urban planning the cost minimizing mass flux $\flux$ between $\mu_+$ and $\mu_-$ satisfies $$\urbPlXia^{\varepsilon,a}(\flux) \leq a \Wdone(\mu_+,\mu_-)\,,$$ where\hfill$\displaystyle\Wdone(\mu_+,\mu_-) = \min_{\mu\in\Pi(\mu_+,\mu_-)} \int_{\R^n\times\R^n} |x - y|\, \de\mu(x,y)$\hfill\ \\[.6\baselineskip]
with $\Pi(\mu_+,\mu_-) = \{\mu\in\fbm(\R^n\times\R^n) \ : \ \mu(A \times \R^n) = \mu_+(A),\ \mu(\R^n \times B) = \mu_-(B)\text{ for $A,B$ Borel }\subset\R^n\}$ denotes the Wasserstein distance.

Indeed, if $\mu_+,\mu_-$ are discrete measures, let $\pi = \sum_{{x \in \spt\mu_+, y \in \spt\mu_-}} c_{x,y}\delta_{(x,y)}\in\Pi(\mu_+,\mu_-)$ be the optimal transport plan for the Wasserstein distance between $\mu_+,\mu_-$.
Now introduce the transport path with vertices $V(G) = \spt\mu_+\cup\spt\mu_-$, edges $E(G) = \left\{e_{x,y} \ : \ x\in\spt\mu_+, \ y\in\spt\mu_-\right\}$, and weights $w(e_{x,y})=c_{x,y}$, where $e_{x,y}$ be 
the oriented line segment from $x$ to $y$. Obviously,
\begin{multline*}
 \urbPlXia^{\varepsilon,a}(\flux) \leq \urbPlXia^{\varepsilon,a}(G) = \sum_{e \in E(G)} \min\{aw(e),w(e)+\varepsilon\}l(e) \leq \sum_{e \in E(G)} aw(e)l(e)\\
 = a\sum_{{x \in \spt\mu_+ , y \in \spt\mu_-}} c_{x,y}|x-y| = a \Wdone(\mu_+,\mu_-)\,.
\end{multline*}
For non-discrete measures $\mu_+,\mu_-$, let $(G_n,\mu_+^n,\mu_-^n) \weakstarto (\flux,\mu_+,\mu_-)$ be an approximating graph sequence. We then finally have
\begin{displaymath}
 \urbPlXia^{\varepsilon,a}(\flux) \leq \liminf_{n \to \infty} \urbPlXia^{\varepsilon,a}(\flux_{G_n}) \leq \liminf_{n \to \infty} a\Wdone(\mu_+^n,\mu_-^n) = a\Wdone(\mu_+,\mu_-)\,.
\end{displaymath}
\end{remark}

\subsection{Energy scaling for a simple, two-dimensional situation}
Given an initial and final measure $\mu_+$ and $\mu_-$, we look for the optimal mass fluxes between $\mu_+$ and $\mu_-$, that is, for minimizers of
\begin{align*}
\urbPlEn^{\varepsilon,a,\mu_+,\mu_-}[\flux]&=\begin{cases}\urbPlXia^{\varepsilon,a}(\flux)&\text{if }\flux\text{ is a mass flux between $\mu_+$ and $\mu_-$},\\\infty&\text{else,}\end{cases}\\
\brTptEn^{\varepsilon,\mu_+,\mu_-}[\flux]&=\begin{cases}\XiaEn^{\varepsilon}(\flux)&\text{if }\flux\text{ is a mass flux between $\mu_+$ and $\mu_-$},\\\infty&\text{else.}\end{cases}
\end{align*}
To derive an energy scaling law we here consider the geometrically simple, two-dimensional ($n=2$) situation with source and sink
\begin{equation}\label{eqn:sourceSink}
\mu_+:=\mu_0=\lebesgue^1\restr[0,\ell]\times\{0\}\,,
\qquad
\mu_-:=\mu_1=\lebesgue^1\restr[0,\ell]\times\{1\}
\end{equation}
(cf.\ Figure\,\ref{fig:setting} left).
It is not difficult to see that the minimum urban planning and branched transport cost for $\varepsilon=0$ are given by
\begin{align*}
\urbPlEn^{*,a,\mu_0,\mu_1}&:=\min_{\flux}\urbPlEn^{0,a,\mu_0,\mu_1}[\flux]=\Wdone(\mu_0,\mu_1)=\ell\,,\qquad\\
\brTptEn^{*,\mu_0,\mu_1}&:=\min_{\flux}\brTptEn^{0,\mu_0,\mu_1}[\flux]\,=\Wdone(\mu_0,\mu_1)=\ell\,.
\end{align*}
Indeed, both urban planning and branched transport cost are known to be bounded below by the Wasserstein distance (see \cite[Lemma 4.2]{Xia-Optimal-Paths} for the case of branched transport;
for the case of urban planning the equivalent formulation of \cite{BrWi15-equivalent} allows the same proof),
and a minimizing flux achieving the above value is given by $\flux=\genfrac(){0pt}{}{0}{1}\lebesgue^2\restr[0,\ell]\times[0,1]$.
Our aim is to prove via a novel relation to image inpainting that urban planning and branched transport satisfy the following energy scaling law.

\begin{theorem}[Energy scaling]\label{thm:scalingUrbPlan}
There are constants $C_1,C_2,\varepsilon_0>0$ independent of $\varepsilon,a,\ell$
such that for $\varepsilon<\min(\varepsilon_0,\ell^3)$ and $a>1$ we have
\begin{equation*}
\begin{array}{rcl}
C_1\ell\min\{a-1,\varepsilon^{\frac23}\}
\;\leq&\min_\flux\urbPlEn^{\varepsilon,a,\mu_0,\mu_1}[\flux]-\urbPlEn^{*,a,\mu_0,\mu_1}
&\leq\; C_2\ell\min\{a-1,\varepsilon^{\frac23}\}\,,\\[1ex]
C_1\ell\varepsilon|\log\varepsilon|
\;\leq&\min_\flux\brTptEn^{\varepsilon,\mu_0,\mu_1}[\flux]-\brTptEn^{*,\mu_0,\mu_1}
&\leq\; C_2\ell\varepsilon|\log\varepsilon|\,.
\end{array}
\end{equation*}
\end{theorem}

The proof of the above result has already been performed in \cite{BrWi15-micropatterns}.
In particular, the upper bound was proven by providing a construction $\flux$ with the desired energy scaling.
The lower bound was obtained by a technique common in pattern analysis.
One aim of this article is to establish the lower bound via a completely different route,
by finding a convex relaxation of the problem and proving a lower bound via convex duality.
The essential ingredient here is the identification of a relation between transport network optimization and image inpainting.
The remainder of the article derives this relation and provides the lower bound proof as well as some numerical simulations as applications.

\section{Optimal networks in 2D via convex optimization}\label{sec:MumfordShah}
In this section we will transform the urban planning and the branched transport problem in two dimensions
to variants of the Mumford--Shah segmentation problem from image processing.
Since for the Mumford--Shah segmentation problem there are well-known convex reformulations via so-called functional lifting,
the urban planning and branched transport problem in two dimensions each come with a convex formulation as well.
While we show that the convex problems provide a lower bound to the original network optimization, we do not examine when they are equivalent.
This would require to prove a limsup inequality which probably involves a difficult technical construction outside the scope of this article (cf. Remark \ref{rem:upper_bound} and Section \ref{sec:discussion}).
We will use the convex formulations to put forward a proof of the lower bound in Theorem\,\ref{thm:scalingUrbPlan} which is solely based on convex duality
as well as to perform numerical network optimization that cannot get stuck in local minima.

\subsection{Bijection between fluxes and images}
Denoting by $B_r(A)$ the open $r$-neighbourhood of a set $A\subset\R^2$,
let $\Omega\subset\R^2$ and $V\subset\R^2$ be open bounded convex domains with $\Omega\subset\subset V$; for simplicity we shall assume
\begin{equation*}
V\supset B_1(\Omega)\,.
\end{equation*}
We will later only consider the case where the initial and final measure of the urban planning or branched transport problem are concentrated on the boundary $\partial\Omega$.
For the specific choice \eqref{eqn:sourceSink} we may for instance take $\Omega=(0,\ell)\times(0,1)$ and $V=B_1(\Omega)$.
Now let $u\in\BV(V)$ represent a grey value image, where $\BV(V)$ denotes the space of scalar functions of bounded variation on $V$.
The gradient of $u$ can be decomposed into a part continuous with respect to the two-dimensional Lebesgue measure $\lebesgue^2$,
a jump part concentrated on the discontinuity set $S_u$ of $u$, and a Cantor part $D_c u$ \cite[(3.89) and (3.90)]{AmFuPa00},
\begin{equation*}
 Du=\nabla u\lebesgue^2\restr V+[u]\nu\hdone\restr S_u + D_c u\,.
\end{equation*}
Here $\nabla u$ is the Radon--Nikodym derivative of the Lebesgue-continuous part with respect to $\lebesgue^2$,
$\nu$ is the unit normal to the (rectifiable) discontinuity set $S_u$, and $[u]$ is the jump in function values across $S_u$ in direction $\nu$.

\begin{definition}[Flux corresponding to an image]
Let $u\in\BV(V)$ be an image. The \emph{flux $\flux_u\in\rca(V;\R^2)$ associated to the image $u$} is defined by
\begin{equation*}
\flux_u=Du^\perp=\nabla u^\perp\lebesgue^2\restr V+[u]\nu^\perp\hdone\restr S_u + D_c u^\perp\,,
\end{equation*}
where superscript $\perp$ denotes a counterclockwise rotation by $\frac\pi2$.
\end{definition}

Note that $\flux_u$ is divergence-free in the distributional sense (in $V$).
Indeed, let $\phi\in\cont_c^\infty(V)$ be a smooth test function with compact support, then
\begin{equation*}
\int_V\phi\,\de(\dv\flux_u)
=-\int_V\nabla\phi\cdot\de\flux_u
=\int_V\nabla\phi^\perp\cdot\de Du
=-\int_V\dv(\nabla\phi^\perp)u\,\de\lebesgue^2
=0
\end{equation*}
due to $\dv(\nabla\phi^\perp)=\frac{\partial^2}{\partial x_2\partial x_1}\phi-\frac{\partial^2}{\partial x_1\partial x_2}\phi=0$.

As mentioned before, we consider the case of $\mu_+$ and $\mu_-$ being any measures of equal mass with
\begin{equation}\label{eqn:measureSpt}
\spt\mu_+\subset\partial\Omega\,,\qquad
\spt\mu_-\subset\partial\Omega\,.
\end{equation}
Without loss of generality we may assume $\partial\Omega$ to contain the origin and to lie in the right halfplane.
The function $\gamma:[0,\hd^1(\partial\Omega))\to\partial\Omega$ shall parameterize $\partial\Omega$ counterclockwise by arclength, starting in $\gamma(0)=0$,
and its image of $[0,t)$ shall be denoted $\partial\Omega_t=\gamma([0,t))$.
Finally, let us introduce the orthogonal projection
\begin{equation*}
\pi_{\partial\Omega}:\R^2\to\partial\Omega\,,
\quad
x\mapsto\arg\min\big\{|x-y|\,:\,y\in\partial\Omega\big\}
\end{equation*}
(in case of nonuniqueness we just arbitrarily pick one closest point to $x$ on $\partial\Omega$).

\begin{definition}[Admissible fluxes and images]
Given finite Borel measures $\mu_+,\mu_-$ satisfying \eqref{eqn:measureSpt}, we define the function
\begin{equation}\label{eqn:bdyCondition}
u(\mu_+,\mu_-):\R^2\to\R,\,x\mapsto(\mu_+-\mu_-)(\partial\Omega_{\gamma^{-1}(\pi_{\partial\Omega}(x))})
\end{equation}
(note that $u(\mu_+,\mu_-)$ is constant along rays orthogonal to $\partial\Omega$).

The sets of \emph{admissible fluxes} and of \emph{admissible images} are given as
\begin{align}
\nonumber\mathcal{A}_\flux(\mu_+,\mu_-)&=\{\flux\in\rca(V;\R^2)\ :\ \spt\flux\subset\overline\Omega,\,\dv\flux=\mu_+-\mu_-\}\,,\\
\label{eq:definition_of_Au}\mathcal{A}_u(\mu_+,\mu_-)&=\{u\in\BV(V)\ :\ u=u(\mu_+,\mu_-)\text{ on }V\setminus\overline\Omega\}.
\end{align}
\end{definition}

The following lemma establishes a one-to-one relation between admissible fluxes and admissible images.
\begin{lemma}\label{thm:imagesAndFluxes}
The mapping $u\mapsto\flux_u\restr\overline\Omega$ from $\mathcal A_u(\mu_+,\mu_-)$ to $\mathcal A_\flux(\mu_+,\mu_-)$ is one-to-one.
In particular, for any $\flux\in\mathcal A_\flux(\mu_+,\mu_-)$ there is a unique image $u_\flux\in\mathcal A_u(\mu_+,\mu_-)$ with $\flux=\flux_{u_\flux}\restr\overline\Omega$.
\end{lemma}

\begin{definition}[Image corresponding to a flux]\label{def:imagesAndFluxes}
Given a flux $\flux\in\mathcal{A}_\flux(\mu_+,\mu_-)$, the image $u_\flux$ from the previous lemma is called the \emph{image corresponding to $\flux$}.
\end{definition}

\begin{proof}[Proof of Lemma\,\ref{thm:imagesAndFluxes}]
We first verify that indeed $\flux_u\restr\overline\Omega\in\mathcal A_\flux(\mu_+,\mu_-)$.
Denoting by $n$ the unit outward normal to $\partial(V\setminus\overline{\Omega})$ and by $T$ the trace operator on $\BV(V\setminus\overline{\Omega})$ \cite[\S\,5.3]{EvansGariepy}, for $\phi\in\cont_c^\infty(V)$ we have
\begin{multline*}
\int_V\phi\,\de\dv(\flux_u\restr\overline\Omega)
=-\int_{\overline\Omega}\nabla\phi\cdot\de\flux_u
= \int_{\overline\Omega}\nabla\phi^\perp\cdot\de Du
= \int_{V}\nabla\phi^\perp\cdot\de Du - \int_{V\setminus\overline{\Omega}}\nabla\phi^\perp\cdot\de Du\\
=-\int_{V} u\dv(\nabla\phi^\perp)\,\de\lebesgue^2
 +\int_{V\setminus \overline\Omega} u\dv(\nabla\phi^\perp)\,\de\lebesgue^2
 -\int_{\partial\Omega} Tu \nabla\phi^\perp\cdot n\,\de\hdone
= -\int_{\partial\Omega}Tu\nabla\phi^\perp\cdot n\,\de\hdone\\
=-\int_{[0,\hd^1(\partial\Omega))}Tu(\gamma(t))\tfrac\de{\de t}\phi(\gamma(t))\,\de t
= \int_{[0,\hd^1(\partial\Omega))}\phi(\gamma(t))\,\de\tfrac\de{\de t}Tu(\gamma(t))
= \int_{\partial\Omega}\phi\,\de(\mu_+-\mu_-)
= \int_V\phi\,\de(\mu_+-\mu_-)\,,
\end{multline*}
where we have used $Tu(\gamma(t))=u(\mu_+,\mu_-)\circ\gamma(t)$ with $\frac\de{\de t}u(\mu_+,\mu_-)\circ\gamma=\mu_+-\mu_-$.

Next, let $\flux\in\mathcal A_\flux(\mu_+,\mu_-)$ be given.
We extend this flux to a flux $\tilde\flux$ on $\R^2$ as follows,
\begin{equation*}
\tilde\flux=\flux+(Du(\mu_+,\mu_-)^\perp)\restr(V\setminus\overline\Omega)\,.
\end{equation*}
It is straightforward to check $\dv\tilde\flux=0$ in the distributional sense. Indeed, for $\phi \in \smooth(\R^2)$ and $n$ the inner normal to $\partial\Omega$, we have
\begin{multline*}
\int_{\R^2} \nabla\phi\cdot\de\tilde\flux(x)
= \int_{\overline{\Omega}} \nabla\phi\cdot\de\flux(x)
+ \int_{\R^2\setminus\overline\Omega} \nabla\phi\cdot\de Du(\mu_+,\mu_-)^\perp
=-\int_{\overline{\Omega}} \phi\, \de\dv\flux
- \int_{\R^2\setminus\overline\Omega} \nabla\phi^\perp\cdot\de Du(\mu_+,\mu_-)\\
= \int_{\overline{\Omega}} \phi\, \de(\mu_--\mu_+)
- \int_{\partial\Omega} Tu \nabla\phi^\perp\cdot n\,\de\hdone
+ \int_{\R^2\setminus\overline\Omega}u(\mu_+,\mu_-) \dv(\nabla\phi^\perp)\,\de\lebesgue^2\\
= \int_{\overline{\Omega}} \phi\, \de(\mu_--\mu_+) + \int_{\partial\Omega} \phi\, \de(\mu_+-\mu_-) = 0\,.
\end{multline*}

Next, we mollify $\tilde\flux$ using a smooth Dirac sequence $\eta_\delta\in\cont_c^\infty(\R^2)$ with support in the $\delta$-ball around the origin, $\spt\eta_\delta\subset B_\delta(0)$,
\begin{equation*}
\flux_\delta=\tilde\flux*\eta_\delta\,.
\end{equation*}
For any $\delta>0$, $\flux_\delta$ is smooth and divergence-free
(with a slight misuse of notation we interpret $\flux_\delta$ as the density with respect to $\lebesgue^2$).
Now define
\begin{equation*}
\mathcal G_\delta=-\flux_\delta^\perp\,,\qquad
u_\delta(x)=(u(\mu_+,\mu_-)*\eta_\delta)(-1,0)+\int_{\gamma}\mathcal G_\delta\cdot\de\gamma\,,
\end{equation*}
where $\gamma$ is any Lipschitz-continuous path connecting $(-1,0)^T$ with $x$.
Since $\mathcal G_\delta$ is curl-free (recall that the curl of a two-dimensional vector field is defined via the curl of the embedded vector field in $\R^3$) and thus conservative, $u_\delta$ is independent of the particular path $\gamma$ chosen.
We have
\begin{equation*}
\nabla u_\delta
=\mathcal G_\delta
=-\flux_\delta^\perp
\weakstarto-\tilde\flux^\perp
\end{equation*}
in $\rca(V)$ as $\delta\to0$.
Furthermore, due to $\nabla u_\delta=-\flux_\delta^\perp=(Du(\mu_+,\mu_-))*\eta_\delta=\nabla(u(\mu_+,\mu_-)*\eta_\delta)$ in $V\setminus B_\delta(\Omega)$ and $u_\delta(-1,0)=(u(\mu_+,\mu_-)*\eta_\delta)(-1,0)$ we have $u_\delta=u(\mu_+,\mu_-)*\eta_\delta$ on $V\setminus B_\delta(\Omega)$
so that $$\sup_{x\in V\setminus B_\delta(\Omega)}|u_\delta(x)|\leq \sup_{x\in\R^2}|u(\mu_+,\mu_-)(x)|\leq\mu_+(\partial\Omega)+\mu_-(\partial\Omega)\,.$$ Thus,
$\|u_\delta\|_{L^1(V)}$ is uniformly bounded by virtue of Poincar\'e's inequality and the uniform boundedness of $\|\nabla u_\delta\|_{L^1(V)}$.
Combining the above properties, we see that as $\delta\to0$ a subsequence of $u_\delta$ converges weakly-$*$ in $\BV(V)$ against some function $u_\flux\in\BV(V)$
which satisfies all conditions from the definition of $\mathcal{A}_u(\mu_+,\mu_-)$ as well as $Du_\flux^\perp\restr\overline\Omega=\flux$
(recall that $u_n$ is said to converge weakly-$*$ in $\BV(V)$ against $u$, $u_n \weakstarto u$, if $u_n \to u$ in $L^1(V)$ and $Du_n \weakstarto Du$ in $\rca(V)$ \cite[Def.\,3.11]{AmFuPa00}).

The uniqueness of $u_\flux$ follows in the standard way:
Assume, $u_i\in\mathcal A_u(\mu_+,\mu_-)$, $i=1,2$, satisfy $\flux=Du_i^\perp\restr\overline\Omega$.
Then $u=u_2-u_1$ satisfies $Du=0$ as well as $u=0$ on $V\setminus\Omega$ so that $u\equiv0$.
\end{proof}

\subsection{Image inpainting problem induced by 2D network optimization}\label{sec:inpainting}
Let us now introduce functionals $\tilde \urbPlEn^{\varepsilon,a,\mu_+,\mu_-},\tilde\brTptEn^{\varepsilon,\mu_+,\mu_-}$ acting on images rather than mass fluxes.

\begin{definition}[Mumford--Shah-type functionals]
We set $\urbPlImg^{\varepsilon,a},\brTptImg^{\varepsilon}:\BV(V)\to[0,\infty]$,
\begin{align*}
\urbPlImg^{\varepsilon,a}[u]
&=a\int_{\overline\Omega\setminus S_u}|Du|\,\de x+\int_{S_u\cap\overline\Omega} \min\{a[u],[u]+\varepsilon\}\,\de\hdone(x)\,,\\
\brTptImg^{\varepsilon}[u]
&=\int_{S_u\cap\overline\Omega}[u]^{1-\varepsilon}\,\de\hdone(x)+\iota_0((\nabla u\lebesgue^2+D_cu)\restr\overline\Omega)\,,
\end{align*}
where the indicator function is given by
\begin{displaymath}
 \iota_0(\mu) = \begin{cases}
                 0 & \text{if}\ \mu = 0\\
                 \infty & \text{otherwise}.
                \end{cases}
\end{displaymath}
Furthermore, for $\mu_+,\mu_-\in\fbm(\R^2)$ of equal mass and satisfying \eqref{eqn:measureSpt},
we define the following functionals on $\BV(V)$,
\begin{align*}
\tilde\urbPlEn^{\varepsilon,a,\mu_+,\mu_-}[u]&=\begin{cases}
\urbPlImg^{\varepsilon,a}[u]&\text{if }u\in\mathcal A_u(\mu_+,\mu_-),\\
\infty&\text{else,}
\end{cases}\\
\tilde\brTptEn^{\varepsilon,\mu_+,\mu_-}[u]&=\begin{cases}
\brTptImg^{\varepsilon}[u]&\text{if }u\in\mathcal A_u(\mu_+,\mu_-),\\
\infty&\text{else.}
\end{cases}
\end{align*}
\end{definition}

Both functionals can be viewed as image inpainting functionals.
Indeed, the image $u$ is prescribed on $V\setminus\overline\Omega$, and inside $\overline\Omega$ it is to be reconstructed as the minimizer of $\tilde \urbPlEn^{\varepsilon,a,\mu_+,\mu_-}$ or $\tilde \brTptEn^{\varepsilon,\mu_+,\mu_-}$.
Furthermore, both functionals are of Mumford--Shah-type in that their integrand is convex in $Du$ away from the jump set $S_u$, but subadditive on the jump set.

In this section we prove the following theorem, expressing the relation between the network costs and the image processing functionals.

\begin{theorem}[Lower bound on transport problems via Mumford--Shah functionals]\label{thm:equivalence_of_transport_problems_with_mumford-shah}
For any flux $\flux\in\mathcal{A}_\flux(\mu_+,\mu_-)$ and the corresponding image $u_\flux\in\mathcal{A}_u(\mu_+,\mu_-)$ (see Definition \ref{def:imagesAndFluxes}) we have
\begin{equation*}
\urbPlEn^{\varepsilon,a,\mu_+,\mu_-}[\flux]\geq\tilde \urbPlEn^{\varepsilon,a,\mu_+,\mu_-}[u_\flux]\,,\qquad
\brTptEn^{\varepsilon,\mu_+,\mu_-}[\flux]\geq\tilde\brTptEn^{\varepsilon,\mu_+,\mu_-}[u_\flux]\,.
\end{equation*}
\end{theorem}

\begin{remark}[Upper bound]\label{rem:upper_bound}
We believe that the opposite inequality holds as well, but we do not attempt a proof here.
It would imply that the urban planning and branched transport problem can in fact be formulated
as the minimization of the above variants of the Mumford--Shah image segmentation functional.
In view of Lemma\,\ref{thm:equivalenceDiscrete} below, a proof would require to show a certain limsup inequality,
namely that the functionals $\urbPlImg^{\varepsilon,a,\mu_+,\mu_-}$ and $\brTptImg^{\varepsilon,\mu_+,\mu_-}$ are actually the sequentially weakly-$*$ lower semi-continuous envelopes
of their restrictions to images $u_{\flux_G}$ corresponding to discrete graphs $G$.
For the standard Mumford--Shah functional this is implied by \cite[Theorem 0.6]{DMMS92}.
For other variants of the Mumford--Shah functional one would have to carefully study the regularity of the singular set of minimizers.
The difficulties one would encounter are of similar type as those which characterize the regularity theory for the standard Mumford--Shah minimizers (which still is only partial despite several decades of investigation).
In addition, existence theorems like \cite[Theorem 4.1]{Amb90} or \cite[Theorem 5.24]{AmFuPa00} do not apply in our case.
\end{remark}

In order to prove Theorem \ref{thm:equivalence_of_transport_problems_with_mumford-shah} we need some preliminary results.

\begin{lemma}[Reduction of measure support]\label{thm:supportReduction}
Let $\mu_+,\mu_-\in\fbm(\R^2)$ of equal mass, satisfying \eqref{eqn:measureSpt},
and let $\J$ either denote the branched transport functional $\XiaEn^{\varepsilon}$ or the urban planning functional $\urbPlXia^{\varepsilon,a}$.
We have
\begin{displaymath}
 \J(\flux) = \inf\left\{\liminf_{n\to\infty} \J(G_n) \ : \ (\flux_{G_n},\mu_+^n,\mu_-^n) \weakstarto (\flux,\mu_+,\mu_-), \,\mu_+^n,\mu_-^n \text{ satisfy \eqref{eqn:measureSpt}}\right\}\,.
\end{displaymath}
\end{lemma}

\begin{proof}
Let $(\flux_{G_n},\mu_+^n,\mu_-^n) \weakstarto (\flux,\mu_+,\mu_-)$ be an approximating graph sequence. Then, we have $\J(\flux)=\lim_{n\to\infty}\J(G_n)$.
We will construct another approximating graph sequence $(\flux_{\tilde G_n},\tilde\mu_+^n,\tilde\mu_-^n)$ with $\spt\tilde\mu_+^n,\spt\tilde\mu_-^n \subseteq\partial\Omega$ and the same limit energy $\J(\flux)=\lim_{n\to\infty}\J(\tilde G_n)$.
Indeed, for $\mu_+^n=\sum_{i=1}^{N_n}a_i^n\delta_{x_i^n}$ and $\mu_-^n=\sum_{i=1}^{M_n}b_i^n\delta_{y_i^n}$ we set
\begin{equation*}
\tilde\mu_+^n=\sum_{i=1}^{N_n}a_i^n\delta_{\pi_{\partial\Omega}(x_i^n)}\,,\qquad
\tilde\mu_-^n=\sum_{i=1}^{M_n}b_i^n\delta_{\pi_{\partial\Omega}(y_i^n)}\,,
\end{equation*}
where $\pi_{\partial\Omega}:\R^2\to\partial\Omega$ is the closest point projection (as before, in case of non-uniqueness an arbitrary closest point is chosen).
Since $\tilde\mu_+^n$ is the measure $\mu_+^n$ projected onto $\partial\Omega$, which is Lipschitz regular and contains the support of $\mu_+$, it is clear that $\mu_+^n\weakstarto\mu_+$ implies $\tilde\mu_+^n\weakstarto\mu_+$ (analogously for $\tilde\mu_-^n$).
Let further $G_+^n$ be a $\J$-minimizing transport path between $\tilde\mu_+^n$ and $\mu_+^n$ and $G_-^n$ a $\J$-minimizing transport path between $\mu_-^n$ and $\tilde\mu_-^n$.
Such transport paths exist, since optimal mass fluxes between discrete measures are known to be discrete graphs
(in the urban planning case this follows from the same argument as in \cite[proof of Prop.\,4.3.1]{BrWi15-equivalent};
in the branched transport case this is contained in \cite[Thm.\,4.7]{Bernot-Caselles-Morel-Structure-Branched} or \cite[Thm.\,4.10]{Xia-Interior-Regularity}).
Moreover, note that by Remark\,\ref{rem:urban_planning_bounded_by_a_times_wasserstein_one} the minimum urban planning cost is bounded above by $a$ times the Wasserstein-1 distance $\Wdone$, which metrises weak-$*$ convergence,
and that the minimum branched transport cost is continuous with respect to weak-$*$ convergence \cite[Lemma 4.1]{Xia-Optimal-Paths}.
Thus we have
\begin{equation*}
\J(G_\pm^n)\to0\text{ as }n\to\infty\text{ due to }\Wdone(\mu_\pm^n,\tilde\mu_\pm^n)\to0\,.
\end{equation*}
We now set $\tilde G_n=(V(G_n)\cup V(G_+^n)\cup V(G_-^n),E(G_n)\cup E(G_+^n)\cup E(G_-^n))$.
It is clear that $\tilde G_n$ is a transport path between $\tilde\mu_+^n$ and $\tilde\mu_-^n$.
Furthermore, $\J(G_\pm^n)\to_{n\to\infty}0$ implies $\flux_{G_\pm^n}\weakstarto0$ and thus
\begin{equation*}
\flux_{\tilde G_n}=\flux_{G_n}+\flux_{G_+^n}+\flux_{G_-^n}\weakstarto\flux\text{ as }n\to\infty.
\end{equation*}
Finally, as desired we have
\begin{equation*}
\J(\tilde G_n)\leq\J(G_n)+\J(G_+^n)+\J(G_-^n)\to\J(\flux)
\quad\text{and}\quad
\J(\tilde G_n)\geq\J(G_n)\to\J(\flux)
\end{equation*}
as $n\to\infty$.
\end{proof}

\begin{lemma}[Equivalence for discrete fluxes]\label{thm:equivalenceDiscrete}
Let $\mu_+,\mu_-\in\fbm(\R^2)$ of equal mass, satisfying \eqref{eqn:measureSpt},
and let $G$ be a transport path between $\mu_+$ and $\mu_-$. Then
\begin{equation*}
\urbPlXia^{\varepsilon,a}[G]=\urbPlImg^{\varepsilon,a}[u_{\flux_{G}}]\,,\qquad
\XiaEn^{\varepsilon}[G]=\brTptImg^{\varepsilon}[u_{\flux_{G}}]\,
\end{equation*}
for $\flux_G$ from Definition\,\ref{def:discrete_mass_flux} and $u_{\flux_{G}}$ from Definition\,\ref{def:imagesAndFluxes}.
\end{lemma}
\begin{proof}
For simplicity let us abbreviate $u\equiv u_{\flux_G}$.
By Lemma\,\ref{thm:imagesAndFluxes}, $u\in\mathcal{A}_u(\mu_+,\mu_-)$ as well as
\begin{equation*}
Du\restr\overline\Omega=-\sum_{e\in E(G)}w(e)(\hdone\restr e)\hat e^\perp
\end{equation*}
for $\hat e$ the unit vector parallel to the edge $e$.
In particular, $u$ is piecewise constant with a discontinuity set
$S_{u}\cap\overline\Omega=\bigcup_{e\in E(G)}e$ on which it has jump size $[u]=w(e)$.
Inserting this into the urban planning and branched transport energy, we obtain
\begin{align*}
\urbPlImg^{\varepsilon,a}[u]
&=\int_{S_{u}\cap\overline\Omega} \min\{a[u],[u]+\varepsilon\}\,\de\hdone(x)
=\sum_{e\in E(G)}l(e)\min\{aw(e),w(e)+\varepsilon\}
=\urbPlXia^{\varepsilon,a}[G]\,,\\
\brTptImg^{\varepsilon}[u]
&=\int_{S_{u}\cap\overline\Omega} [u]^{1-\varepsilon}\,\de\hdone(x)
=\sum_{e\in E(G)}l(e)w(e)^{1-\varepsilon}
=\XiaEn^{\varepsilon}[G]\,,
\end{align*}
the desired result.
\end{proof}

Next let us define the set of functions
\begin{equation}\label{eq:definition_of_H}
H=\{u\in\BV(V)\,:\,u(x)=u(\pi_{\partial\Omega}(x))\,\forall x\in V\setminus\overline\Omega\}\,.
\end{equation}
Note that $H$ is closed in $\BV(V)$ with respect to weak-$*$ convergence (even with respect to $L^1(V)$-convergence).
Indeed, let $u_n\in H$ with $u_n\weakstarto u$ in $\BV(V)$, then a subsequence converges pointwise almost everywhere against $u$,
and since the $u_n$ are constant along rays orthogonal to $\partial\Omega$, the pointwise limit must be so as well.

\begin{lemma}[Lower semi-continuity]\label{thm:lsc}
$\urbPlImg^{\varepsilon,a}$ and $\brTptImg^\varepsilon$ are sequentially weakly-$*$ lower semi-continuous on $H$.
\end{lemma}
\begin{proof}
For $u\in H$, $n\in\N$, and any measurable set $B\subset V$ define
\begin{align*}
E_B[u]
&=a|Du|(B\setminus S_u)+\int_{S_u\cap B} \min\{a[u],[u]+\varepsilon\}\,\de\hdone\,,\\
M_B^n[u]
&=\int_{S_u\cap B}\min\{n[u],[u]^{1-\varepsilon}\}\,\de\hdone+n|Du|(B\setminus S_u)\,.
\end{align*}
We have $\urbPlImg^{\varepsilon,a}[u]=E_{\overline\Omega}[u]$ as well as $\brTptImg^\varepsilon[u]=\sup_nM_{\overline\Omega}^n[u]$.
Note that by the monotone convergence theorem one has
\begin{displaymath}
\sup_n\int_{S_u\cap B}\min\{n[u],[u]^{1-\varepsilon}\}\,\de\hdone=\lim_{n\to\infty}\int_{S_u\cap B}\min\{n[u],[u]^{1-\varepsilon}\}\,\de\hdone=\int_{S_u\cap B}[u]^{1-\varepsilon}\,\de\hdone.
\end{displaymath}
Thus it suffices to show the sequential weak-$*$ lower semi-continuity of $E_{\overline\Omega}$ and $M_{\overline\Omega}^n$ for all $n$.

Now let $B_{\delta}=B_\delta(\Omega)$ be the open $\delta$-neighbourhood of $\Omega$.
By \cite[Thm.\,5.4]{AmFuPa00}, $E_{B_\delta}$ and $M_{B_\delta}^n$ are sequentially weakly-$*$ lower semi-continuous on $\BV(B_\delta)$.
Consider a sequence $u_k\weakstarto u$ in $H$ and let $C\in\R$ be an upper bound on its total variation, $C>\|Du_k\|_\rca,\|Du\|_\rca$.
We have
\begin{displaymath}
E_{\overline\Omega}(u)
\leq E_{B_\delta}(u)
\leq\liminf_{k\to\infty}E_{B_\delta}(u_k) \leq\liminf_{k\to\infty}E_{\overline\Omega}(u_k)+a|Du|(B_\delta\setminus\overline\Omega)
\leq\liminf_{k\to\infty}E_{\overline\Omega}(u_k)+2a\delta C\,.
\end{displaymath}
Since $\delta$ is arbitrary, the above implies the sequential weak-$*$ lower semi-continuity of $E_{\overline\Omega}$.
Analogously,
\begin{displaymath}
M_{\overline\Omega}^n(u)
\leq M_{B_\delta}^n(u)
\leq\liminf_{k\to\infty}M_{B_\delta}(u_k)
\leq\liminf_{k\to\infty}M_{\overline\Omega}(u_k)+n|Du|(B_\delta\setminus\overline\Omega)
\leq\liminf_{k\to\infty}M_{\overline\Omega}(u_k)+2n\delta C\,,
\end{displaymath}
again implying the sequential weak-$*$ lower semi-continuity of $M_{\overline\Omega}^n$ by the arbitrariness of $\delta$.
\end{proof}

\begin{remark}[Density of piecewise constant images]\label{rem:density_of_piecewise_constant_images}
Let $\SBV(V)\subset\BV(V)$ denote the space of special functions of bounded variation, that is, those functions in $\BV(V)$ whose derivative has no Cantor part.
Furthermore, let $D\subset\SBV(V)$ denote the set of piecewise constant images $u$ whose discontinuity set is composed of finitely many straight lines.

$D$ is dense in $\BV(V)$ with respect to weak-$*$ convergence.
Indeed, consider a quadrilateral grid $\mathcal G$ over $V$ of sidelength $\frac1n$ and approximate $u\in\BV(V)$ by the function $u_n\in D$
which on each square face of $\mathcal G$ equals the average of $u$ on that face,
\begin{displaymath}
 u_n(y) = \fint_{V_{ij}^n} u \,\de x
 \qquad
 \text{for all }y\in V_{ij}^n:=([\tfrac{i}{n},\tfrac{i+1}{n})\times[\tfrac{j}{n},\tfrac{j+1}{n}))\cap V\,.
\end{displaymath}

We have $u_n\to u$ strongly in $L^1(V)$, which together with the boundedness of $\|Du_n\|_{\rca}$ implies $u_n\weakstarto u$ up to a subsequence.
\end{remark}

\begin{proof}[Proof of Theorem\,\ref{thm:equivalence_of_transport_problems_with_mumford-shah}]
By Lemmas \ref{thm:imagesAndFluxes}, \ref{thm:supportReduction}, \ref{thm:equivalenceDiscrete}, and \ref{thm:lsc} we have
\begin{align*}
\urbPlXia^{\varepsilon,a}(\flux)
&=\inf\{\liminf_{n\to\infty}\urbPlXia^{\varepsilon,a}(G_n) \ : \ (\flux_{G_n},\mu_+^n,\mu_-^n) \weakstarto (\flux,\mu_+,\mu_-),\,\mu_+^n,\mu_-^n\text{ satisfy \eqref{eqn:measureSpt}}\}\\
&=\inf\{\liminf_{n\to\infty}\urbPlImg^{\varepsilon,a}(u_{\flux_{G_n}}) \ : \ (\flux_{G_n},\mu_+^n,\mu_-^n) \weakstarto (\flux,\mu_+,\mu_-),\,u_{\flux_{G_n}}\in\mathcal{A}_u(\mu_+^n,\mu_-^n)\}\\
&=\inf\{\liminf_{n\to\infty}\urbPlImg^{\varepsilon,a}(u_n) \ : \ (Du_n^\perp,\mu_+^n,\mu_-^n) \weakstarto (Du_\flux^\perp,\mu_+,\mu_-),\,u_n\in D\cap\mathcal{A}_u(\mu_+^n,\mu_-^n)\}\\
&=\inf\{\liminf_{n\to\infty}\urbPlImg^{\varepsilon,a}(u_n) \ : \ u_n \weakstarto u_\flux\text{ in }\BV(V),\,u_n\in D\cap H\}\\
&\geq\urbPlImg^{\varepsilon,a}(u_\flux)
\,,
\end{align*}
the desired result. In particular, the first equality is due to Lemma \ref{thm:supportReduction}, the second to Lemma \ref{thm:equivalenceDiscrete}, the third to Lemma \ref{thm:imagesAndFluxes}, and the fourth to the definition of $\mathcal{A}_u(\mu_+,\mu_-)$ in \eqref{eq:definition_of_Au} and of $H$ in \eqref{eq:definition_of_H}. The final inequality is then due to the lower semi-continuity Lemma \ref{thm:lsc}.

The result for $\tilde\brTptEn^\varepsilon$ is obtained analogously.
\end{proof}

\subsection{Functional lifting of network optimization}\label{sec:lifting}
In a series of articles \cite{PCBC-2009b}, \cite{PCBC-2010} Pock, Cremers, Bischof, and Chambolle showed that the minimization of certain variational problems admits a convex reformulation via so-called functional lifting.
The underlying idea is based on \cite{Alberti-Bouchitte-DalMaso} and \cite{Cha01}, which consider functionals of the form
\begin{equation*}
\J[u]=\int_\Omega g(x,u(x),\nabla u(x))\,\de x+\int_{S_u}\psi(x,u^+,u^-,\nu)\,\de\hdone(x)\,,
\end{equation*}
where $u^+$ and $u^-$ are the function values of $u$ on either side of the discontinuity set $S_u$,
and $g:\Omega\times\R\times\R^2\to[0,\infty]$ and $\psi:\Omega\times\R\times\R\times S^1\to[0,\infty]$.
Introducing the characteristic function of the subgraph of the image $u$,
\begin{equation*}
1_u:\Omega\times\R\to\{0,1\},\,(x,s)\mapsto\begin{cases}1&\text{if }u(x)>s\,,\\0&\text{else,}\end{cases}
\end{equation*}
the authors show
\begin{align*}
\J[u]\geq&\sup_{\phi\in\mathcal K}\int_{\Omega\times\R}\phi\cdot\de D1_u\\
&\text{with }
\mathcal K=\bigg\{\phi=(\phi_x,\phi_s)\in\cont_0^\infty(\Omega\times\R;\R^2\times\R)\ :\\
&\qquad\qquad\quad\phi_s(x,s)\geq g^*(x,s,\phi_x(x,s))\quad\forall(x,s)\in\Omega\times\R\,,\\
&\textstyle\qquad\quad\qquad\left|\int_{s_1}^{s_2}\phi_x(x,s)\,\de s\right|\leq\psi(x,s_1,s_2,\nu)\quad\forall x\in\Omega,s_1<s_2,\nu\in S^1\bigg\}\,.
\end{align*}
Here, $g^*$ denotes the Legendre--Fenchel dual of $g$ with respect to its last argument.
In case of strong duality, even equality holds.
Consequently, instead of minimizing the potentially non-convex functional $\J[u]$ for $u$
one can find a solution to the saddle point problem $\min_{1_u}\sup_{\phi\in\mathcal K}\int_{\Omega\times\R}\phi\cdot\de D 1_u$.

Unfortunately, the set $\{v\!\in\!\BV(\Omega\!\times\!\R;\{0,1\})\ :\ v=1_u\text{ for some }u\!\in\!\BV(\Omega)\}$ of characteristic functions, over which one minimizes, is non-convex.
However, this set can be relaxed to the larger set
\begin{equation*}
\mathcal C=\left\{v\in\BV(\Omega\times\R;[0,1])\ :\
\lim_{s\to-\infty}v(x,s)=1\,,\;\lim_{s\to\infty}v(x,s)=0\right\}\,.
\end{equation*}
With this relaxation,
\begin{equation}\label{eq:convex_formulation}
\inf_{u\in\BV(\Omega)}\J[u]
\geq\inf_{v\in\mathcal C}\sup_{\phi\in\mathcal K}\int_{\Omega\times\R}\phi\cdot\de D v\,,
\end{equation}
where the right-hand side is a convex problem.

For our setting, given $\mu_+,\mu_-\in\fbm(\R^2)$ of equal mass and with \eqref{eqn:measureSpt},
let $1_{u(\mu_+,\mu_-)}$ be the characteristic function of the subgraph of the function defined in \eqref{eqn:bdyCondition}, and let us introduce the sets
\begin{align*}
\tilde{\mathcal C}&=\Big\{v\in\BV(V\times\R;[0,1])\ :\\
&\qquad\lim_{s\to-\infty}v(x,s)=1\,,\;\lim_{s\to\infty}v(x,s)=0\,,\;v=1_{u(\mu_+,\mu_-)}\text{ on }(V\setminus\overline\Omega)\times\R\Big\}\,,\\
\mathcal K_1&=\Big\{\phi=(\phi_x,\phi_s)\in\cont_0^\infty(V\times\R;\R^2\times\R)\,:\\
&\textstyle\qquad\phi_s\geq0,\,|\phi_x|\leq a,\,\left|\int_{s_1}^{s_2}\phi_x(x,s)\,\de s\right|\leq\min\{|s_2-s_1|+\varepsilon,a|s_2-s_1|\}\,\forall x\in V,\,s_1,s_2\in\R\Big\}\,,\\
\mathcal K_2&=\Big\{\phi=(\phi_x,\phi_s)\in\cont_0^\infty(V\times\R;\R^2\times\R)\,:\\
&\textstyle\qquad\phi_s\geq0,\,\left|\int_{s_1}^{s_2}\phi_x(x,s)\,\de s\right|\leq|s_2-s_1|^{1-\varepsilon}\,\forall x\in V,\,s_1,s_2\in\R\Big\}\,.
\end{align*}
Note that $\Kcal_1$ is $\Kcal$ for the choice corresponding to urban planning,
\begin{displaymath}
 g(x,u,p) = ap\,,\qquad \psi(x,u^+,u^-,\nu) = \min\{a|u^+-u^-|,|u^+-u^-|+\varepsilon\}\,,
\end{displaymath}
and that $\Kcal_2$ is $\Kcal$ for the choice corresponding to branched transport,
\begin{displaymath}
 g(x,u,p) = \iota_0(p)\,,\qquad \psi(x,u^+,u^-,\nu) = |u^+-u^-|^{1-\varepsilon}\,.
\end{displaymath}

Applying the previous discussion, we have
\begin{eqnarray*}
 \min_{\flux\in\mathcal A_\flux(\mu_+,\mu_-)}\urbPlEn^{\varepsilon,a,\mu_+,\mu_-}[\flux] \geq &\inf_{u\in\mathcal A_u(\mu_+,\mu_-)}\tilde\urbPlEn^{\varepsilon,a,\mu_+,\mu_-}[u] &\geq \inf_{v\in\tilde{\mathcal C}}\sup_{\phi\in\mathcal K_1}\int_{\overline\Omega\times\R}\phi\cdot\de Dv\,,\\
 \min_{\flux\in\mathcal A_\flux(\mu_+,\mu_-)}\brTptEn^{\varepsilon,\mu_+,\mu_-}[\flux] \geq &\inf_{u\in\mathcal A_u(\mu_+,\mu_-)}\tilde \brTptEn^{\varepsilon,a,\mu_+,\mu_-}[u] &\geq \inf_{v\in\tilde{\mathcal C}}\sup_{\phi\in\mathcal K_2}\int_{\overline\Omega\times\R}\phi\cdot\de Dv\,.
\end{eqnarray*}
The right-hand sides are convex optimization problems in $v$.
Furthermore, in some cases the optimal $v$ can be shown to have the form $1_u$ for the solution u of the generalised Mumford--Shah problem (\cite[Theorem 8.1]{Cha01} for the one-dimensional problem).
The optimal mass flux can then be readily read off as $Du^\perp$.

\subsection{Lower bound on network costs by convex duality}
As an application of the convex reformulation from the previous section
we finally obtain a proof of the lower bound in Theorem\,\ref{thm:scalingUrbPlan} just based on convex duality.
In particular, we show the following.
\begin{theorem}
There exist $c,\varepsilon_0>0$ such that for $n=2$ and $\varepsilon<\varepsilon_0$ we have
\begin{align*}
\min_\flux \urbPlEn^{\varepsilon,a,\mu_0,\mu_1}[\flux]-\urbPlEn^{*,a,\mu_0,\mu_1}&\geq c\ell\min\{\varepsilon^{\frac23},a-1\}\,,\\
\min_\flux \brTptEn^{\varepsilon,\mu_0,\mu_1}[\flux]-\brTptEn^{*,\mu_0,\mu_1}&\geq c\ell\varepsilon|\log\varepsilon|\,.
\end{align*}
\end{theorem}
\begin{proof}
Here we choose $\Omega=(0,\ell)\times(0,1)$ and $V=B_1(\Omega)$.
Note that for $\mu_0,\mu_1$ from \eqref{eqn:sourceSink} we have $u(\mu_0,\mu_1)(x)=x_1$ for $x\in\partial\Omega$.
Swapping the infimum and supremum in the convex formulation we obtain
\begin{align*}
\hspace*{10ex}&\hspace*{-10ex}\min_{\flux\in\mathcal A_\flux(\mu_0,\mu_1)}\urbPlEn^{\varepsilon,a,\mu_0,\mu_1}[\flux]
\geq\sup_{\phi\in\mathcal K_1}\inf_{v\in\tilde{\mathcal C}}\int_{\overline\Omega\times\R}\phi\cdot\de Dv\\
&=\sup_{\phi\in\mathcal K_1}\inf_{v\in\tilde{\mathcal C}}\int_{P}\phi\cdot\nu\,\de\hd^2(x,s)-\int_{\Omega\times\R}v\dv\phi\,\de x\,\de s\\
&=\sup_{\phi\in\mathcal K_1}\int_P\phi\cdot\nu\,\de\hd^2(x,s)-\int_{\Omega\times\R}\max(0,\dv\phi)\,\de x\,\de s,
\end{align*}
where $P=\{(x,s)\in\partial \Omega \times \R\,:\,1_{u(\mu_0,\mu_1)}(x,s)=1\}=\{(x,s)\in\partial \Omega \times \R\,:\,x_1\geq s\}$,
$\nu$ is the unit outward normal to $P$, as well as
\begin{equation*}
\min_{\flux\in\mathcal A_\flux(\mu_0,\mu_1)}\brTptEn^{\varepsilon,\mu_0,\mu_1}[\flux]
\geq\sup_{\phi\in\mathcal K_2}\int_P\phi\cdot\nu\,\de\hd^2(x,s)-\int_{\Omega\times\R}\max(0,\dv\phi)\,\de x\,\de s\,.
\end{equation*}

All that remains to do is to construct a proper test function $\phi$ to be used in the above convex duality estimate.
Define the clockwise circular flow
\begin{equation*}
\theta(x_1,x_2)=\tfrac1r\left(\begin{smallmatrix}x_2\\-x_1\end{smallmatrix}\right)\quad\text{if }r=\left|\left(\begin{smallmatrix}x_2\\-x_1\end{smallmatrix}\right)\right|\leq\tfrac12
\text{ and }\theta(x_1,x_2)=0\text{ else}\,.
\end{equation*}
Obviously, $\dv\theta=0$ in the sense of distributions.
Now we stretch this circular flow horizontally by a factor $2\beta_i>2$ (which will be specified later), $i=1,2$, and divide by two,
\begin{equation*}
\varphi_i(x)=\tfrac12B_i\theta(B_i^{-1}x)\qquad\text{for }B_i=\left(\begin{smallmatrix}2\beta_i&0\\0&1\end{smallmatrix}\right)\,.
\end{equation*}
It is readily checked that also $\dv\varphi_i=0$.
Now define $\phi^i=(\phi^i_x,0)$ as (cf.\,Figure\,\ref{fig:testfield})
\begin{equation*}
\phi^i_x(x,s)=\begin{cases}0&\text{if }s\notin[0,\ell]\,,\\
\varphi_i(x_1-s,x_2)&\text{if }x_2\leq\frac12,\,s\in[0,\ell]\,,\\
-\varphi_i(x_1-s,1-x_2)&\text{else.}\end{cases}
\end{equation*}
Denoting the unit outward normal to $\partial\Omega$ by $n$ and using $\dv\phi^i_x(\cdot,s)=0$ as well as $\phi^i_x((s,x_2),s)=\varphi_i(0,x_2)=\genfrac{(}{)}{0pt}{}{\beta_i}{0}$ for $x_2\in[0,1]$ we can calculate
\begin{multline*}
\int_P\phi^i\cdot\nu\,\de\hd^2(x,s)
=\int_0^\ell \int_{\{x\in\partial\Omega\,:\,x_1\geq s\}}\phi^i_x\cdot n \,\de\hdone(x)\,\de s\\
=\int_0^\ell\left( \int_{[s,\ell]\times[0,1]}\dv\phi^i_x(\cdot,s)\,\de x-\int_{\{s\}\times[0,1]}\phi^i_x(x,s)\cdot\genfrac{(}{)}{0pt}{1}{-1}{0} \,\de\hdone(x)\right)\,\de s
=\beta_i\ell\,,
\end{multline*}
where in the second step we used the divergence theorem.
Now the test functions $\phi^1,\phi^2$ can be approximated arbitrarily well by divergence-free functions in $\cont_0^\infty(V\times\R;\R^2\times\R)$ so that the above convex duality estimates imply
\begin{align*}
\min_{\flux\in\mathcal A_\flux(\mu_0,\mu_1)}\urbPlEn^{\varepsilon,a,\mu_0,\mu_1}[\flux]
&\geq\int_P\phi^1\cdot\nu\,\de\hd^2(x,s)=\beta_1\ell=\urbPlEn^{*,a,\mu_0,\mu_1}+\ell(\beta_1-1)\,,\\
\min_{\flux\in\mathcal A_\flux(\mu_0,\mu_1)}\brTptEn^{\varepsilon,\mu_0,\mu_1}[\flux]
&\geq\int_P\phi^2\cdot\nu\,\de\hd^2(x,s)=\beta_2\ell=\brTptEn^{*,\mu_0,\mu_1}+\ell(\beta_2-1)\,,
\end{align*}
where $\beta_1,\beta_2$ will be chosen such that
\begin{align*}
\left|\int_{s_1}^{s_2}\phi^1_x(x,s)\,\de s\right|&\leq\min\{|s_2-s_1|+\varepsilon,a|s_2-s_1|\}\quad\forall x\in V, s_1,s_2\in\R\,,\\
\left|\int_{s_1}^{s_2}\phi^2_x(x,s)\,\de s\right|&\leq|s_2-s_1|^{1-\varepsilon}\quad\forall x\in V, s_1,s_2\in\R
\end{align*}
as well as $\beta_1\leq a$ in order to satisfy $|\phi^1_x|\leq a$.
We are going to show that $\beta_1=\min(1+c\varepsilon^{\frac23},a)$ and $\beta_2=1+c\varepsilon|\log\varepsilon|$ are admissible choices for some $c>0$, which concludes the proof.

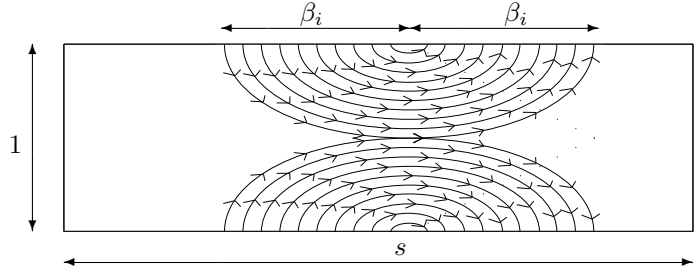
\begin{figure}
\centering
\setlength{\unitlength}{1.3\linewidth}
\begin{picture}(.4,.13)
\put(.02,.02){\includegraphics[width=.4\unitlength,trim=0 120 0 120,clip]{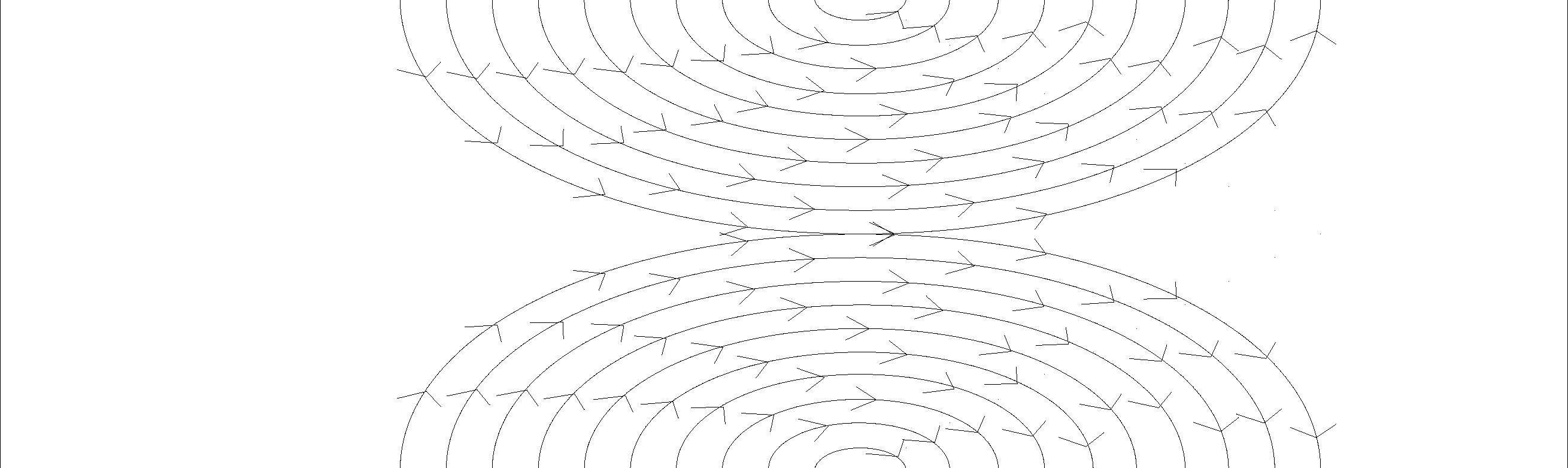}}
\put(.02,.02){\line(1,0){.4}}
\put(.02,.14){\line(1,0){.4}}
\put(.02,.02){\line(0,1){.12}}
\put(.42,.02){\line(0,1){.12}}
\put(.12,0){\vector(-1,0){.1}}
\put(.12,0){\vector(1,0){.3}}
\put(.22,-.02){$\ell$}
\put(.23,.005){$s$}
\put(0,.07){\vector(0,-1){.05}}
\put(0,.07){\vector(0,1){.07}}
\put(-.015,.07){$1$}
\put(.15,.15){\vector(-1,0){.03}}
\put(.15,.15){\vector(1,0){.09}}
\put(.17,.155){$\beta_i$}
\put(.27,.15){\vector(-1,0){.03}}
\put(.27,.15){\vector(1,0){.09}}
\put(.3,.155){$\beta_i$}
\end{picture}
\caption{Sketch of the flow $\phi^i$ in cross-section $\Omega \times \{s\}$.}
\label{fig:testfield}
\end{figure}

Fix a position $x=(x_1,x_2)$, where due to symmetry it suffices to consider $x_2\in[0,\frac12]$.
Let us abbreviate $\hat s=s-x_1$ and $r_i(\hat s)=|B_i^{-1}\genfrac(){0pt}{}{\hat s}{x_2}| = \sqrt{x_2^2+(\frac{\hat s}{2\beta_i})^2}$.
We have
\begin{equation*}
|\phi^i(x,s)|
=|\varphi_i(-\hat s,x_2)|
=\tfrac1{r_i(\hat s)}\sqrt{\beta_i^2x_2^2+(\tfrac{\hat s}{4\beta_i})^2}
=\sqrt{\tfrac{\beta_i^2+(\hat s/(4x_2\beta_i))^2}{1+(\hat s/(2x_2\beta_i))^2}}
\end{equation*}
where $\phi^i$ is nonzero.
Note that for $\bar s_i=4x_2\beta_i\sqrt{(\beta_i^2-1)/3}$,
\begin{equation*}
|\phi^i(x,s)|\leq\begin{cases}\beta_i&\text{if }\hat s\in[-\bar s_i,\bar s_i]\,,\\1&\text{else,}\end{cases}
\end{equation*}
thus in particular $|\phi^i(x,s)|\leq\beta_i$.

We first consider the case $i=1$, in which we have (assuming $s_2\geq s_1$ without loss of generality)
\begin{multline*}
\left|\int_{s_1}^{s_2}\phi^1(x,s)\,\de s\right|-|s_2-s_1|
\leq\int_{s_1}^{s_2}|\phi^1(x,s)|-1\,\de s\\
\leq2\bar s_1(\beta_1-1)
\leq4\beta_1\sqrt{\tfrac{\beta_1^2-1}3}(\beta_1-1)
\leq4\sqrt{\tfrac{\beta_1^2-1}3}(\beta_1^2-1)\,,
\end{multline*}
which is indeed smaller than $\varepsilon$, as required, if we choose $\beta_1=\min(1+\frac14(\frac{3}2)^{\frac23}\varepsilon^{\frac23},a)$.
Likewise, this choice of $\beta_1$ also satisfies
\begin{equation*}
\left|\int_{s_1}^{s_2}\phi^1(x,s)\,\de s\right|-a|s_2-s_1|
\leq2\bar s_1(\beta_1-a)\leq0
\end{equation*}
as required.

For $i=2$, assume first $\hat s_1,\hat s_2\in[-\bar s_i,\bar s_i]$. In that case,
\begin{equation*}
\left|\int_{s_1}^{s_2}\phi^2(x,s)\,\de s\right|/|s_2-s_1|
\leq\int_{s_1}^{s_2}|\phi^2(x,s)|\,\de s/|s_2-s_1|
\leq\beta_2\,.
\end{equation*}
Choosing $\beta_2=1+\frac{\varepsilon|\log\varepsilon|}2$, this is smaller than $(2\bar s_2)^{-\varepsilon}\leq|s_2-s_1|^{-\varepsilon}$, as required.
Indeed, for $\varepsilon\leq1$, $\beta_2=1+\frac{\varepsilon|\log\varepsilon|}2\leq\sqrt\varepsilon^{-\varepsilon}\leq\big(4\sqrt{\frac{\varepsilon|\log\varepsilon|}3}\big)^{-\varepsilon}\leq(2\bar s_2)^{-\varepsilon}$.
It remains to show $\left|\int_{s_1}^{s_2}\phi(x,s)\,\de s\right|\leq|s_2-s_1|^{1-\varepsilon}$ for $\hat s_1$ or $\hat s_2$ outside $[-\bar s_2,\bar s_2]$.
In the next paragraph we will show
\begin{equation}\label{eq:postponed_inequality}
 |\phi^2(x,s)|\leq(1-\varepsilon)|2\hat s|^{-\varepsilon} \quad \text{for $\hat s\in[-\beta_2,\beta_2]\setminus[-\bar s_2,\bar s_2]$},
\end{equation}
which then yields for a given $\Delta s=|s_2-s_1|$ (exploiting that $|\phi^2|$ decreases to both sides of $\hat s=0$)
\begin{multline*}
\left|\int_{s_1}^{s_2}\phi^2(x,s)\,\de s\right|
\leq\int_{s_1}^{s_2}|\phi^2(x,s)|\,\de s
\leq\int_{x_1-\frac{\Delta s}2}^{x_1+\frac{\Delta s}2}|\phi^2(x,s)|\,\de s\\
\leq\begin{cases}
(\Delta s)^{1-\varepsilon}&\text{if }\frac{\Delta s}2\leq\bar s_2\,,\\
(2\bar s_2)^{1-\varepsilon}+2\int_{\bar s_2}^{\frac{\Delta s}2}(1-\varepsilon)|2\hat s|^{-\varepsilon}\,\de\hat s
=(\Delta s)^{1-\varepsilon}&\text{else,}
\end{cases}
\end{multline*}
as required.

To prove inequality \eqref{eq:postponed_inequality},
let us abbreviate $t=\frac{(1-\varepsilon)^{1/\varepsilon}}2\approx\frac1{2e}$.
Obviously, $$(1-\varepsilon)|2\hat s|^{-\varepsilon}=|\tfrac t{\hat s}|^{-\varepsilon}\geq1\geq|\phi^2(x,s)|\quad\text{ for }\hat s\in[-t,t]\setminus[-\bar s_2,\bar s_2]\,,$$
while for $\hat s\in[-\beta_2,\beta_2]\setminus[t,t]$ we have
\begin{equation*}
|\phi^2(x,s)|
\leq\sqrt{\tfrac{\beta_2^4+t^2/4}{\beta_2^2+t^2}}
\leq(1-\varepsilon)|2\beta_2|^{-\varepsilon}
\leq(1-\varepsilon)|2\hat s|^{-\varepsilon}\,,
\end{equation*}
where the middle inequality holds for $\varepsilon$ small enough due to $\sqrt{\tfrac{\beta_2^4+t^2/4}{\beta_2^2+t^2}}\to\sqrt{\frac{(4e)^2+1}{(4e)^2+4}}<1$ and $(1-\varepsilon)|2\beta_2|^{-\varepsilon}\to1$ as ${\varepsilon\to0}$.
\end{proof}

\subsection{Numerical optimization of urban planning and branched transport networks}
The proposed convex reformulation of the branched transport and urban planning model as variants of the Mumford--Shah segmentation problem allows a numerical network optimization that cannot get stuck in local minima due to the convexity.
We shall use a finite difference discretization of the right-hand side in \eqref{eq:convex_formulation} (similarly to \cite{PCBC-2009}) and apply a simple primal-dual algorithm to
numerically find the saddle point.

In the following, for simplicity we assume without loss of generality a rectangular domain $V\subset\R^2$ with bottom left corner at the origin.
We discretize the domain $V\times\mathbb{R}$ of the lifted function $v$ by a finite three-dimensional $n\times m\times p$ grid
\begin{equation*}
\mathcal{G} = \{(ih_1,jh_2,lh_s): \ i=0,\ldots,n, \ j=0,\ldots,m, \ l=0,\ldots,p\}\,,
\end{equation*}
where $h_1,h_2,h_s>0$ denote the grid size in each direction. Hence, the discrete counterparts of $v\in\BV(V\times\mathbb{R};[0,1])$ and $\phi\in C_0^{\infty}(V\times\mathbb{R};\mathbb{R}^2\times\mathbb{R})$
are given by $v^h:\mathcal{G}\rightarrow[0,1]$ and $\phi^h:\mathcal{G}\rightarrow\mathbb{R}^2\times\mathbb{R}$. For the sake of simplicity, for every $(ih_1,jh_2,lh_s)\in\mathcal{G}$, we write
$v_{ijl}^h=v^h(ih_1,jh_2,lh_s)$ and $\phi_{ijl}^h=\phi^h(ih_1,jh_2,lh_s)$. Defining the finite difference gradient operator $D=(D_1,D_2,D_s)^T$ by
\begin{equation*}
(D_1v^h)_{ijl} = \tfrac{v_{i+1,j,l}^h-v_{i,j,l}^h}{h_1}\,,\quad
(D_2v^h)_{ijl} = \tfrac{v_{i,j+1,l}^h-v_{i,j,l}^h}{h_2}\,,\quad
(D_sv^h)_{ijl} = \tfrac{v_{i,j,l+1}^h-v_{i,j,l}^h}{h_s}\,,
\end{equation*}
the discretized form of saddle point problem \eqref{eq:convex_formulation} reads
\begin{equation*}
\min_{v^h\in \mathcal{C}^h}\max_{\phi^h\in \mathcal{K}^h} \sum_{i,j,l}\phi_{ijl}^h(Dv^h)_{ijl}\,,
\end{equation*}
in which the discrete versions of the convex sets $\tilde{\mathcal{C}}$ and $\mathcal{K}=\mathcal K_1$ (for urban planning) or $\mathcal K_2$ (for branched transport) are given by
\begin{align*}
\mathcal{C}^h&=\Big\{v^h:\mathcal{G}\rightarrow[0,1]:\ v_{ij0}^h=1, \ v_{ijp}^h=0 \ \forall i,j, \ v^h=1_{u(\mu_+,\mu_-)}^h\text{ on }\mathcal{G}\setminus(\overline\Omega\times\R)\Big\}\,,\\
\mathcal{K}^h_1&=\Big\{\phi^h=(\phi^h_x,\phi^h_s):\mathcal{G}\rightarrow\mathbb{R}^2\times\mathbb{R}:\\
&\textstyle\qquad \phi^h_s\geq 0, \ |\phi^h_x|\leq a, \ |h_s\sum_{l=s_1}^{s_2}(\phi^h_x)_{ijl}| \leq \min\{|s_2-s_1|+\varepsilon, a|s_2-s_1|\} \ \forall i,j,s_1<s_2 \Big\}\,,\\
\mathcal{K}^h_2&\textstyle=\Big\{\phi^h=(\phi^h_x,\phi^h_s):\mathcal{G}\rightarrow\mathbb{R}^2\times\mathbb{R}:\,\phi^h_s\geq 0, \ |h_s\sum_{l=s_1}^{s_2}(\phi^h_x)_{ijl}| \leq |s_2-s_1|^{1-\varepsilon} \ \forall i,j,s_1<s_2 \Big\}\,.
\end{align*}
Above, $1_{u(\mu_+,\mu_-)}^h$ denotes the restriction of $1_{u(\mu_+,\mu_-)}$ onto the grid $\mathcal{G}$.

The discretized saddle point problem is solved using the primal-dual algorithm from \cite{ChPo10}, which alternatingly performs a primal gradient descent and a dual gradient ascent step with step length $\tau$ (respectively $\sigma$), accompanied by an overrelaxation with parameter $\theta$. Denoting the $k$\textsuperscript{th} approximation to $v^h$ and $\phi^h$ by $v^k$ and $\phi^k$, respectively, the iterative algorithm reads
\begin{equation*}
\begin{cases}
v^{k+1} = P_{\mathcal{C}^h}(v^k-\tau D^*\phi^k)\,, \\
\overline{v}^{k+1} = v^{k+1}+\theta(v^{k+1}-v^k)\,, \\
\phi^{k+1} = P_{\mathcal{K}^h}(\phi^k+\sigma D\overline{v}^{k+1})\,,
\end{cases}
\end{equation*}
starting from an arbitrary initialization $v^0$, $\phi^0$.
The orthogonal projection $P_{\mathcal{K}^h}$ onto $\mathcal{K}^h$ is performed by Dykstra's method \cite{BoDy86},
while the projection $P_{\mathcal{C}^h}$ onto the set $\mathcal{C}^h$ is straightforward.
When $v^k$ has sufficiently converged, it is typically close to binary. The rounding of $v^k$ to $\{0,1\}$ is then interpreted as the sought function $1_u$, and the optimal flux is identified as $Du^\perp$.

\begin{figure}
\centering
\setlength{\unitlength}{\linewidth}
\begin{picture}(1.2,.37)(.015,0)
\put(-.08,.002){\includegraphics[width=1.18\textwidth]{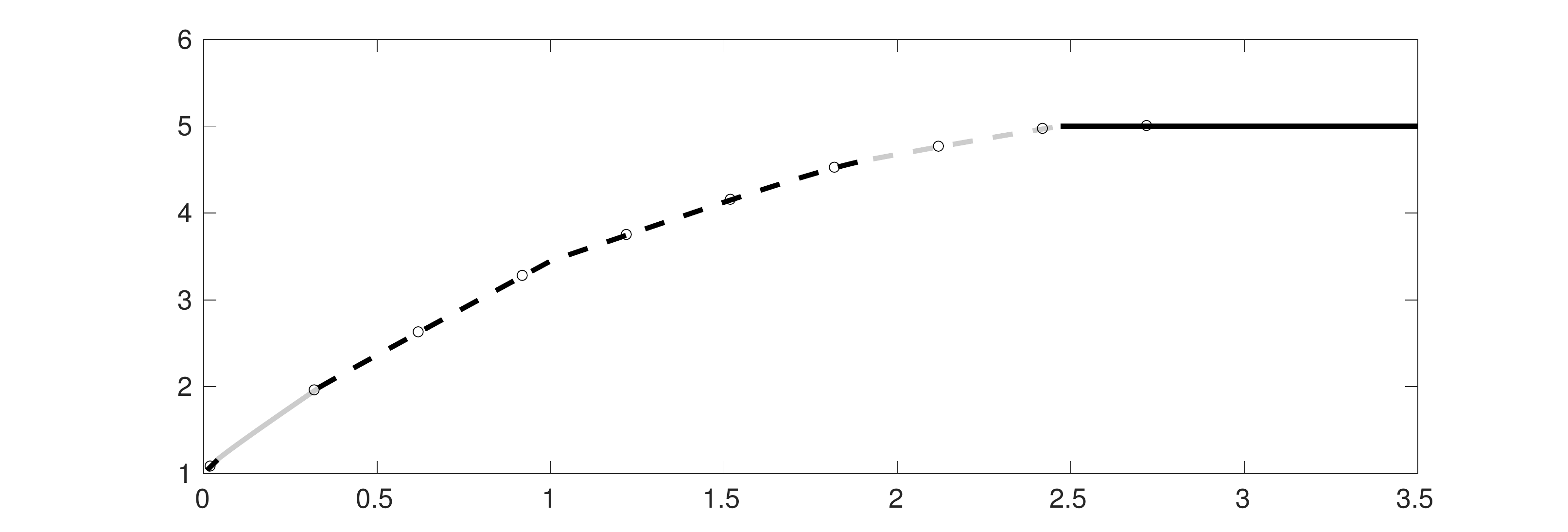}}
\put(.52,.002){\large $\varepsilon$}
\put(.03,.14){\begin{rotate}{90} \large $\underset{\flux}{\min}\ E^{\varepsilon,a}(\flux)$ \end{rotate}}
\put(.076,.073){\circlearound{1}}
\put(.144,.126){\circlearound{2}}
\put(.22,.17){\circlearound{3}}
\put(.3,.215){\circlearound{4}}
\put(.38,.245){\circlearound{5}}
\put(.458,.27){\circlearound{6}}
\put(.535,.295){\circlearound{7}}
\put(.614,.312){\circlearound{8}}
\put(.694,.325){\circlearound{9}}
\put(.772,.33){\tikz[baseline]\node[draw,shape=circle,scale=0.5,anchor=base]{10};}

\put(.62,.065){
\begin{tikzpicture}[scale=.7]\tikzset{every node/.style={inner sep=2pt]}}
  \filldraw (0.15,2.75) circle (1.5pt) node[align=left, above] {\tiny +} -- (0.15,1.4) circle (1.5pt) node[align=left, below] {-};
  \filldraw (0.55,2.75) circle (1.5pt) node[align=left, above] {\tiny +} -- (0.55,1.4) circle (1.5pt) node[align=left, below] {-};
  \filldraw (0.95,2.75) circle (1.5pt) node[align=left, above] {\tiny +} -- (0.95,1.4) circle (1.5pt) node[align=left, below] {-};
  \filldraw (1.35,2.75) circle (1.5pt) node[align=left, above] {\tiny +} -- (1.35,1.4) circle (1.5pt) node[align=left, below] {-};
  \draw [line width=1.7pt] (.11,0.95) -- (1.39,0.95);
\end{tikzpicture}
}
\put(.71,.065){
\begin{tikzpicture}[scale=.7]\tikzset{every node/.style={inner sep=2pt]}}
  \filldraw (0.15,2.75) circle (1.5pt) node[align=left, above] {\tiny +} -- (0.35,2.4);
  \filldraw (0.55,2.75) circle (1.5pt) node[align=left, above] {\tiny +} -- (0.35,2.4);
  \draw (0.35,2.4) -- (0.35,1.75);
  \filldraw (0.15,1.4) circle (1.5pt) node[align=left, below] {-} -- (0.35,1.75);
  \filldraw (0.55,1.4) circle (1.5pt) node[align=left, below] {-} -- (0.35,1.75);
  \filldraw (0.95,2.75) circle (1.5pt) node[align=left, above] {\tiny +} -- (1.15,2.4);
  \filldraw (1.35,2.75) circle (1.5pt) node[align=left, above] {\tiny +} -- (1.15,2.4);
  \draw (1.15,2.4) -- (1.15,1.75);
  \filldraw (0.95,1.4) circle (1.5pt) node[align=left, below] {-} -- (1.15,1.75);
  \filldraw (1.35,1.4) circle (1.5pt) node[align=left, below] {-} -- (1.15,1.75);
  \draw [lightgray,line width=1.7pt] (.11,0.95) -- (1.39,0.95);
\end{tikzpicture}
}
\put(.8,.065){
\begin{tikzpicture}[scale=.7]\tikzset{every node/.style={inner sep=2pt]}}
  \filldraw (0.15,2.75) circle (1.5pt) node[align=left, above] {\tiny +};
  \filldraw (0.55,2.75) circle (1.5pt) node[align=left, above] {\tiny +};
  \filldraw (0.95,2.75) circle (1.5pt) node[align=left, above] {\tiny +};
  \filldraw (1.35,2.75) circle (1.5pt) node[align=left, above] {\tiny +};
  \filldraw (0.15,1.4) circle (1.5pt) node[align=left, below] {-};
  \filldraw (0.55,1.4) circle (1.5pt) node[align=left, below] {-};
  \filldraw (0.95,1.4) circle (1.5pt) node[align=left, below] {-};
  \filldraw (1.35,1.4) circle (1.5pt) node[align=left, below] {-};
  \draw (.15,2.75) -- (.35,2.6); 
  \draw (.55,2.75) -- (.35,2.6);
  \draw (.95,2.75) -- (1.15,2.6);
  \draw (1.35,2.75) -- (1.15,2.6);
  \draw (.15,1.4) -- (.35,1.55); 
  \draw (.55,1.4) -- (.35,1.55);
  \draw (.95,1.4) -- (1.15,1.55);
  \draw (1.35,1.4) -- (1.15,1.55);
  \draw (.75,2.3) -- (.75,1.85);
  \draw (.35,2.6) -- (.75,2.3);
  \draw (1.15,2.6) -- (.75,2.3);
  \draw (.35,1.55) -- (.75,1.85);
  \draw (1.15,1.55) -- (.75,1.85);  
  \draw [dashed,line width=1.7pt] (.105,0.95) -- (1.45,0.95);
\end{tikzpicture}
}
\put(.89,.065){
\begin{tikzpicture}[scale=.7]\tikzset{every node/.style={inner sep=2pt]}}
  \filldraw (0.15,2.75) circle (1.5pt) node[align=left, above] {\tiny +} -- (.75,2.3);
  \filldraw (0.55,2.75) circle (1.5pt) node[align=left, above] {\tiny +} -- (.75,2.3);
  \filldraw (0.95,2.75) circle (1.5pt) node[align=left, above] {\tiny +} -- (.75,2.3);
  \filldraw (1.35,2.75) circle (1.5pt) node[align=left, above] {\tiny +} -- (.75,2.3);
  \filldraw (0.15,1.4) circle (1.5pt) node[align=left, below] {-} -- (.75,1.85);
  \filldraw (0.55,1.4) circle (1.5pt) node[align=left, below] {-} -- (.75,1.85);
  \filldraw (0.95,1.4) circle (1.5pt) node[align=left, below] {-} -- (.75,1.85);
  \filldraw (1.35,1.4) circle (1.5pt) node[align=left, below] {-} -- (.75,1.85);
  \draw (.75,2.3) -- (.75,1.85);
  \draw [dashed,lightgray,line width=1.7pt] (.105,0.95) -- (1.45,0.95);
\end{tikzpicture}
}
\put(.62,.055){\line(1,0){.36}}
\put(.62,.175){\line(1,0){.36}}
\put(.62,.055){\line(0,1){.12}}
\put(.98,.055){\line(0,1){.12}}

\end{picture}
\setlength{\unitlength}{.8\linewidth}
\begin{picture}(.96,.47)
\put(.01,.26){\includegraphics[width=0.18\unitlength]{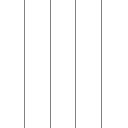}} \put(.01,.43){\circlearound{1}} \put(.06,.235){\small $\varepsilon=0.02$}
\put(.044,.44){\circle*{0.008}}
\put(.081,.44){\circle*{0.008}}
\put(.116,.44){\circle*{0.008}}
\put(.153,.44){\circle*{0.008}}
\put(.044,.26){\circle*{0.008}}
\put(.081,.26){\circle*{0.008}}
\put(.116,.26){\circle*{0.008}}
\put(.153,.26){\circle*{0.008}}
\put(.2,.26){\includegraphics[width=0.18\unitlength]{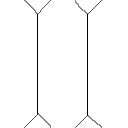}} \put(.2,.43){\circlearound{2}} \put(.25,.235){\small $\varepsilon=0.32$}
\put(.234,.44){\circle*{0.008}}
\put(.271,.44){\circle*{0.008}}
\put(.306,.44){\circle*{0.008}}
\put(.343,.44){\circle*{0.008}}
\put(.234,.26){\circle*{0.008}}
\put(.271,.26){\circle*{0.008}}
\put(.306,.26){\circle*{0.008}}
\put(.343,.26){\circle*{0.008}}
\put(.39,.26){\includegraphics[width=0.18\unitlength]{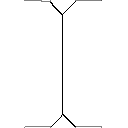}} \put(.39,.43){\circlearound{3}} \put(.44,.235){\small $\varepsilon=0.62$}
\put(.424,.44){\circle*{0.008}}
\put(.461,.44){\circle*{0.008}}
\put(.496,.44){\circle*{0.008}}
\put(.533,.44){\circle*{0.008}}
\put(.424,.26){\circle*{0.008}}
\put(.461,.26){\circle*{0.008}}
\put(.496,.26){\circle*{0.008}}
\put(.533,.26){\circle*{0.008}}
\put(.58,.26){\includegraphics[width=0.18\unitlength]{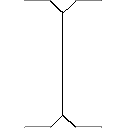}} \put(.58,.43){\circlearound{4}} \put(.63,.235){\small $\varepsilon=0.92$}
\put(.614,.44){\circle*{0.008}}
\put(.651,.44){\circle*{0.008}}
\put(.686,.44){\circle*{0.008}}
\put(.723,.44){\circle*{0.008}}
\put(.614,.26){\circle*{0.008}}
\put(.651,.26){\circle*{0.008}}
\put(.686,.26){\circle*{0.008}}
\put(.723,.26){\circle*{0.008}}
\put(.77,.26){\includegraphics[width=0.18\unitlength]{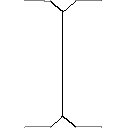}} \put(.77,.43){\circlearound{5}} \put(.82,.235){\small $\varepsilon=1.22$}
\put(.804,.44){\circle*{0.008}}
\put(.841,.44){\circle*{0.008}}
\put(.876,.44){\circle*{0.008}}
\put(.913,.44){\circle*{0.008}}
\put(.804,.26){\circle*{0.008}}
\put(.841,.26){\circle*{0.008}}
\put(.876,.26){\circle*{0.008}}
\put(.913,.26){\circle*{0.008}}
\put(.01,.03){\includegraphics[width=0.18\unitlength]{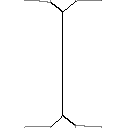}} \put(.01,.2){\circlearound{6}} \put(.06,0.005){\small $\varepsilon=1.52$}
\put(.044,.21){\circle*{0.008}}
\put(.081,.21){\circle*{0.008}}
\put(.116,.21){\circle*{0.008}}
\put(.153,.21){\circle*{0.008}}
\put(.044,.03){\circle*{0.008}}
\put(.081,.03){\circle*{0.008}}
\put(.116,.03){\circle*{0.008}}
\put(.153,.03){\circle*{0.008}}
\put(.2,.03){\includegraphics[width=0.18\unitlength]{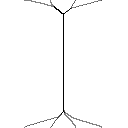}} \put(.2,.2){\circlearound{7}} \put(.25,0.005){\small $\varepsilon=1.82$}
\put(.234,.21){\circle*{0.008}}
\put(.271,.21){\circle*{0.008}}
\put(.306,.21){\circle*{0.008}}
\put(.343,.21){\circle*{0.008}}
\put(.234,.03){\circle*{0.008}}
\put(.271,.03){\circle*{0.008}}
\put(.306,.03){\circle*{0.008}}
\put(.343,.03){\circle*{0.008}}
\put(.39,.03){\includegraphics[width=0.18\unitlength]{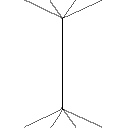}} \put(.39,.2){\circlearound{8}} \put(.44,0.005){\small $\varepsilon=2.12$}
\put(.424,.21){\circle*{0.008}}
\put(.461,.21){\circle*{0.008}}
\put(.496,.21){\circle*{0.008}}
\put(.533,.21){\circle*{0.008}}
\put(.424,.03){\circle*{0.008}}
\put(.461,.03){\circle*{0.008}}
\put(.496,.03){\circle*{0.008}}
\put(.533,.03){\circle*{0.008}}
\put(.58,.03){\includegraphics[width=0.18\unitlength]{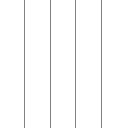}} \put(.58,.2){\circlearound{9}} \put(.63,0.005){\small $\varepsilon=2.42$}
\put(.614,.21){\circle*{0.008}} 
\put(.651,.21){\circle*{0.008}}
\put(.686,.21){\circle*{0.008}}
\put(.723,.21){\circle*{0.008}}
\put(.614,.03){\circle*{0.008}}
\put(.651,.03){\circle*{0.008}}
\put(.686,.03){\circle*{0.008}}
\put(.723,.03){\circle*{0.008}}
\put(.77,.03){\includegraphics[width=0.18\unitlength]{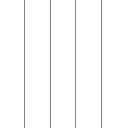}} \put(.77,.2){\tikz[baseline]\node[draw,shape=circle,scale=0.5,anchor=base]{10};} \put(.82,0.005){\small $\varepsilon=2.72$}
\put(.804,.21){\circle*{0.008}}
\put(.841,.21){\circle*{0.008}}
\put(.876,.21){\circle*{0.008}}
\put(.913,.21){\circle*{0.008}}
\put(.804,.03){\circle*{0.008}}
\put(.841,.03){\circle*{0.008}}
\put(.876,.03){\circle*{0.008}}
\put(.913,.03){\circle*{0.008}}
\end{picture}
\caption{Parameter study for urban planning.
Top: Plot of the manually computed minimal energy for different values of $\varepsilon$ and fixed $a=5$. The line type indicates the optimal network topology.
Bottom: Numerically computed optimal fluxes for evenly spaced values of $\varepsilon$ in the same range. The numerically obtained network topologies match the predicted ones except for example $\textcircled{\small 9}$.}
\label{fig:EnergyPlotUP}
\end{figure}

\begin{figure}
\centering
\setlength{\unitlength}{\linewidth}
\begin{picture}(1.2,.37)(.015,0)
\put(-.08,.002){\includegraphics[width=1.18\textwidth]{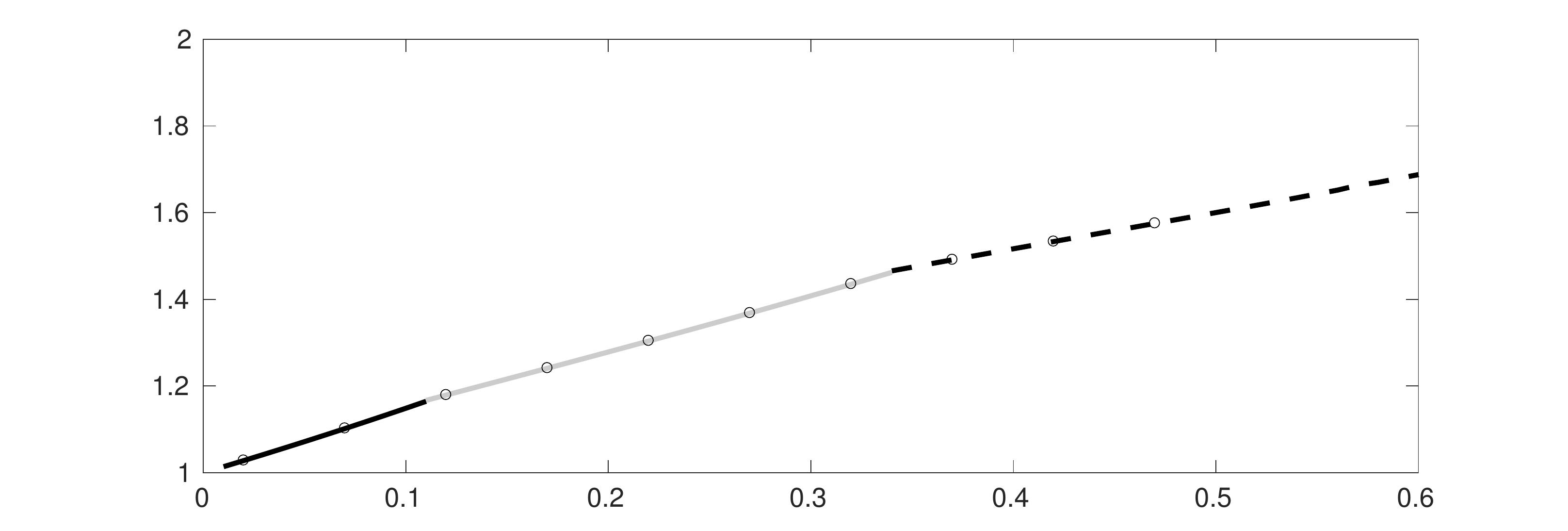}}
\put(.52,.002){\large $\varepsilon$}
\put(.01,.14){\begin{rotate}{90} \large $\underset{\flux}{\min}\ M^{\varepsilon}(\flux)$ \end{rotate}}
\put(.095,.255){
\begin{tikzpicture}[scale=.7]\tikzset{every node/.style={inner sep=2pt]}}
  \filldraw (0.15,2.75) circle (1.5pt) node[align=left, above] {\tiny +} -- (0.15,1.4) circle (1.5pt) node[align=left, below] {-};
  \filldraw (0.55,2.75) circle (1.5pt) node[align=left, above] {\tiny +} -- (0.55,1.4) circle (1.5pt) node[align=left, below] {-};
  \filldraw (0.95,2.75) circle (1.5pt) node[align=left, above] {\tiny +} -- (0.95,1.4) circle (1.5pt) node[align=left, below] {-};
  \filldraw (1.35,2.75) circle (1.5pt) node[align=left, above] {\tiny +} -- (1.35,1.4) circle (1.5pt) node[align=left, below] {-};
  \draw [line width=1.7pt] (.11,0.95) -- (1.39,0.95);
\end{tikzpicture}
}
\put(.185,.255){
\begin{tikzpicture}[scale=.7]\tikzset{every node/.style={inner sep=2pt]}}
  \filldraw (0.15,2.75) circle (1.5pt) node[align=left, above] {\tiny +} -- (0.35,2.4);
  \filldraw (0.55,2.75) circle (1.5pt) node[align=left, above] {\tiny +} -- (0.35,2.4);
  \draw (0.35,2.4) -- (0.35,1.75);
  \filldraw (0.15,1.4) circle (1.5pt) node[align=left, below] {-} -- (0.35,1.75);
  \filldraw (0.55,1.4) circle (1.5pt) node[align=left, below] {-} -- (0.35,1.75);
  \filldraw (0.95,2.75) circle (1.5pt) node[align=left, above] {\tiny +} -- (1.15,2.4);
  \filldraw (1.35,2.75) circle (1.5pt) node[align=left, above] {\tiny +} -- (1.15,2.4);
  \draw (1.15,2.4) -- (1.15,1.75);
  \filldraw (0.95,1.4) circle (1.5pt) node[align=left, below] {-} -- (1.15,1.75);
  \filldraw (1.35,1.4) circle (1.5pt) node[align=left, below] {-} -- (1.15,1.75);
  \draw [lightgray,line width=1.7pt] (.11,0.95) -- (1.39,0.95);
\end{tikzpicture}
}
\put(.275,.255){
\begin{tikzpicture}[scale=.7]\tikzset{every node/.style={inner sep=2pt]}}
  \filldraw (0.15,2.75) circle (1.5pt) node[align=left, above] {\tiny +};
  \filldraw (0.55,2.75) circle (1.5pt) node[align=left, above] {\tiny +};
  \filldraw (0.95,2.75) circle (1.5pt) node[align=left, above] {\tiny +};
  \filldraw (1.35,2.75) circle (1.5pt) node[align=left, above] {\tiny +};
  \filldraw (0.15,1.4) circle (1.5pt) node[align=left, below] {-};
  \filldraw (0.55,1.4) circle (1.5pt) node[align=left, below] {-};
  \filldraw (0.95,1.4) circle (1.5pt) node[align=left, below] {-};
  \filldraw (1.35,1.4) circle (1.5pt) node[align=left, below] {-};
  \draw (.15,2.75) -- (.35,2.6); 
  \draw (.55,2.75) -- (.35,2.6);
  \draw (.95,2.75) -- (1.15,2.6);
  \draw (1.35,2.75) -- (1.15,2.6);
  \draw (.15,1.4) -- (.35,1.55); 
  \draw (.55,1.4) -- (.35,1.55);
  \draw (.95,1.4) -- (1.15,1.55);
  \draw (1.35,1.4) -- (1.15,1.55);
  \draw (.75,2.3) -- (.75,1.85);
  \draw (.35,2.6) -- (.75,2.3);
  \draw (1.15,2.6) -- (.75,2.3);
  \draw (.35,1.55) -- (.75,1.85);
  \draw (1.15,1.55) -- (.75,1.85);  
  \draw [dashed,line width=1.7pt] (.105,0.95) -- (1.45,0.95);
\end{tikzpicture}
}

\put(.095,.365){\line(1,0){.27}}
\put(.095,.24){\line(1,0){.27}}
\put(.095,.24){\line(0,1){.125}}
\put(.365,.24){\line(0,1){.125}}

\put(.09,.077){\circlearound{1}}
\put(.166,.104){\circlearound{2}}
\put(.242,.13){\circlearound{3}}
\put(.32,.151){\circlearound{4}}
\put(.395,.173){\circlearound{5}}
\put(.472,.195){\circlearound{6}}
\put(.55,.211){\circlearound{7}}
\put(.625,.231){\circlearound{8}}
\put(.7,.245){\circlearound{9}}
\put(.776,.26){\tikz[baseline]\node[draw,shape=circle,scale=0.5,anchor=base]{10};}
\end{picture}
\setlength{\unitlength}{.8\linewidth}
\begin{picture}(.96,.47)
\put(.01,.26){\includegraphics[width=0.18\unitlength]{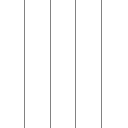}} \put(.01,.43){\circlearound{1}} \put(.06,.235){\small $\varepsilon=0.02$}
\put(.044,.44){\circle*{0.008}}
\put(.081,.44){\circle*{0.008}}
\put(.116,.44){\circle*{0.008}}
\put(.153,.44){\circle*{0.008}}
\put(.044,.26){\circle*{0.008}}
\put(.081,.26){\circle*{0.008}}
\put(.116,.26){\circle*{0.008}}
\put(.153,.26){\circle*{0.008}}
\put(.2,.26){\includegraphics[width=0.18\unitlength]{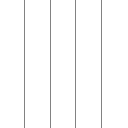}} \put(.2,.43){\circlearound{2}} \put(.25,.235){\small $\varepsilon=0.07$}
\put(.234,.44){\circle*{0.008}}
\put(.271,.44){\circle*{0.008}}
\put(.306,.44){\circle*{0.008}}
\put(.343,.44){\circle*{0.008}}
\put(.234,.26){\circle*{0.008}}
\put(.271,.26){\circle*{0.008}}
\put(.306,.26){\circle*{0.008}}
\put(.343,.26){\circle*{0.008}}
\put(.39,.26){\includegraphics[width=0.18\unitlength]{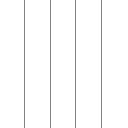}} \put(.39,.43){\circlearound{3}} \put(.44,.235){\small $\varepsilon=0.12$}
\put(.424,.44){\circle*{0.008}}
\put(.461,.44){\circle*{0.008}}
\put(.496,.44){\circle*{0.008}}
\put(.533,.44){\circle*{0.008}}
\put(.424,.26){\circle*{0.008}}
\put(.461,.26){\circle*{0.008}}
\put(.496,.26){\circle*{0.008}}
\put(.533,.26){\circle*{0.008}}
\put(.58,.26){\includegraphics[width=0.18\unitlength]{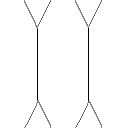}} \put(.58,.43){\circlearound{4}} \put(.63,.235){\small $\varepsilon=0.17$}
\put(.614,.44){\circle*{0.008}}
\put(.651,.44){\circle*{0.008}}
\put(.686,.44){\circle*{0.008}}
\put(.723,.44){\circle*{0.008}}
\put(.614,.26){\circle*{0.008}}
\put(.651,.26){\circle*{0.008}}
\put(.686,.26){\circle*{0.008}}
\put(.723,.26){\circle*{0.008}}
\put(.77,.26){\includegraphics[width=0.18\unitlength]{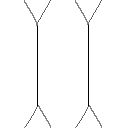}} \put(.77,.43){\circlearound{5}} \put(.82,.235){\small $\varepsilon=0.22$}
\put(.804,.44){\circle*{0.008}}
\put(.841,.44){\circle*{0.008}}
\put(.876,.44){\circle*{0.008}}
\put(.913,.44){\circle*{0.008}}
\put(.804,.26){\circle*{0.008}}
\put(.841,.26){\circle*{0.008}}
\put(.876,.26){\circle*{0.008}}
\put(.913,.26){\circle*{0.008}}
\put(.01,.03){\includegraphics[width=0.18\unitlength]{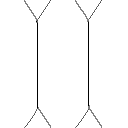}} \put(.01,.2){\circlearound{6}} \put(.06,0.005){\small $\varepsilon=0.27$}
\put(.044,.21){\circle*{0.008}}
\put(.081,.21){\circle*{0.008}}
\put(.116,.21){\circle*{0.008}}
\put(.153,.21){\circle*{0.008}}
\put(.044,.03){\circle*{0.008}}
\put(.081,.03){\circle*{0.008}}
\put(.116,.03){\circle*{0.008}}
\put(.153,.03){\circle*{0.008}}
\put(.2,.03){\includegraphics[width=0.18\unitlength]{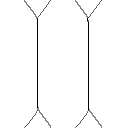}} \put(.2,.2){\circlearound{7}} \put(.25,0.005){\small $\varepsilon=0.32$}
\put(.234,.21){\circle*{0.008}}
\put(.271,.21){\circle*{0.008}}
\put(.306,.21){\circle*{0.008}}
\put(.343,.21){\circle*{0.008}}
\put(.234,.03){\circle*{0.008}}
\put(.271,.03){\circle*{0.008}}
\put(.306,.03){\circle*{0.008}}
\put(.343,.03){\circle*{0.008}}
\put(.39,.03){\includegraphics[width=0.18\unitlength]{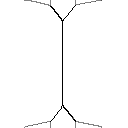}} \put(.39,.2){\circlearound{8}} \put(.44,0.005){\small $\varepsilon=0.37$}
\put(.424,.21){\circle*{0.008}}
\put(.461,.21){\circle*{0.008}}
\put(.496,.21){\circle*{0.008}}
\put(.533,.21){\circle*{0.008}}
\put(.424,.03){\circle*{0.008}}
\put(.461,.03){\circle*{0.008}}
\put(.496,.03){\circle*{0.008}}
\put(.533,.03){\circle*{0.008}}
\put(.58,.03){\includegraphics[width=0.18\unitlength]{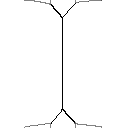}} \put(.58,.2){\circlearound{9}} \put(.63,0.005){\small $\varepsilon=0.42$}
\put(.614,.21){\circle*{0.008}} 
\put(.651,.21){\circle*{0.008}}
\put(.686,.21){\circle*{0.008}}
\put(.723,.21){\circle*{0.008}}
\put(.614,.03){\circle*{0.008}}
\put(.651,.03){\circle*{0.008}}
\put(.686,.03){\circle*{0.008}}
\put(.723,.03){\circle*{0.008}}
\put(.77,.03){\includegraphics[width=0.18\unitlength]{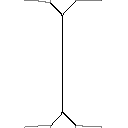}} \put(.77,.2){\tikz[baseline]\node[draw,shape=circle,scale=0.5,anchor=base]{10};} \put(.82,0.005){\small $\varepsilon=0.47$}
\put(.804,.21){\circle*{0.008}}
\put(.841,.21){\circle*{0.008}}
\put(.876,.21){\circle*{0.008}}
\put(.913,.21){\circle*{0.008}}
\put(.804,.03){\circle*{0.008}}
\put(.841,.03){\circle*{0.008}}
\put(.876,.03){\circle*{0.008}}
\put(.913,.03){\circle*{0.008}}
\end{picture}
\caption{Parameter study for branched transport.
Top: Plot of the manually computed minimal energy for different values of $\varepsilon$. The line type indicates the optimal network topology.
Bottom: Numerically computed optimal fluxes for evenly spaced values of $\varepsilon$ in the same range. The numerically obtained network topologies match the predicted ones except for example \textcircled{\small 3}.}
\label{fig:EnergyPlotBT}
\end{figure}

Figures\,\ref{fig:EnergyPlotUP} and \ref{fig:EnergyPlotBT} show numerically optimized urban planning and branched transportation networks or rather fluxes
between four evenly-spaced point sources at the top of a rectangular domain and four evenly spaced point sinks at the bottom of the domain.
For such a simple geometric setting one can still calculate the true optimal solution manually by enumerating all possible network topologies and optimizing their vertex positions by hand.
In both figures, the top graph shows the manually computed optimal energy and corresponding network topology for different parameter values.
Below, numerical solutions are shown that uniformly sample the same parameter range.
Except for example \textcircled{\small 9} in Figure \ref{fig:EnergyPlotUP} and example \textcircled{\small 3} in Figure\,\ref{fig:EnergyPlotBT}, the numerical solutions coincide with the predicted, truly optimal network topology.
The discrepancy between the numerical and the true solution in examples \textcircled{\small 9} and \textcircled{\small 3} (which both lie close to a bifurcation point) may be due to our discretization grid,
which is vertically aligned and thus slightly favours vertical structures, since at fixed grid resolution diagonal structures cannot be represented as exactly.
Note that unlike for other numerical approaches, the shown result cannot just be a suboptimal local optimum, since the numerical method always leads to the global optimum due to convexity.

\begin{figure}
\setlength{\unitlength}{\linewidth}
\begin{picture}(1,.17)
\put(0,.02){\includegraphics[width=0.18\textwidth]{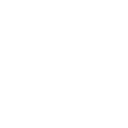}} \put(.065,0){\small $s=0$}
\put(0,.02){\line(0,1){.18}}
\put(0,.02){\line(1,0){.18}}
\put(0,.2){\line(1,0){.18}}
\put(.18,.02){\line(0,1){.18}}
\put(.19,.02){\includegraphics[width=0.18\textwidth]{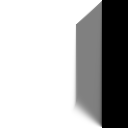}} \put(.25,0){\small $s=\frac{1}{4}$}
\put(.19,.02){\line(0,1){.18}}
\put(.19,.02){\line(1,0){.18}}
\put(.19,.2){\line(1,0){.18}}
\put(.37,.02){\line(0,1){.18}}
\put(.38,.02){\includegraphics[width=0.18\textwidth]{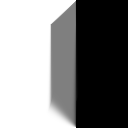}} \put(.44,0){\small $s=\frac{1}{2}$}
\put(.38,.02){\line(0,1){.18}}
\put(.38,.02){\line(1,0){.18}}
\put(.38,.2){\line(1,0){.18}}
\put(.56,.02){\line(0,1){.18}}
\put(.57,.02){\includegraphics[width=0.18\textwidth]{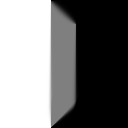}} \put(.63,0){\small $s=\frac{3}{4}$}
\put(.57,.02){\line(0,1){.18}}
\put(.57,.02){\line(1,0){.18}}
\put(.57,.2){\line(1,0){.18}}
\put(.75,.02){\line(0,1){.18}}
\put(.76,.02){\includegraphics[width=0.18\textwidth]{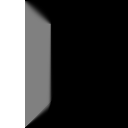}} \put(.825,0){\small $s=1$}
\put(.76,.02){\line(0,1){.18}}
\put(.76,.02){\line(1,0){.18}}
\put(.76,.2){\line(1,0){.18}}
\put(.94,.02){\line(0,1){.18}}
\put(.945,.021){
\begin{tikzpicture}
  \node [shading = axis,rectangle, left color=white, right color=black,shading angle=0, anchor=north, minimum width=.4, minimum height=80] (box) at (current page.north){};
\end{tikzpicture}
}
\put(.98,.01){$0$}
\put(.975,.1){$0.5$}
\put(.98,.19){$1$}
\end{picture}
\caption{Example of a numerical optimization for urban planning, resulting in a non-binary solution $v$ (the images show different cross-sections).
This indicates a lack of tightness of the convex reformulation for the chosen parameters ($a=2.13$ and $\varepsilon=0.5$).
The parameters lie close to a bifurcation at which the truly optimal network topology changes from four vertical pipes to a single tree.}
\label{fig:NonWorkingExample}
\end{figure}

If the proposed convex reformulation of urban planning and branched transport were always tight,
the optimal solution $v$ to \eqref{eq:convex_formulation} would be binary and only take values $0$ or $1$.
However, in some situations this is not the case.
For the case of urban planning, Figure\,\ref{fig:NonWorkingExample} shows a numerically optimized function $v$ which distinctly takes on three values, $0$, $\frac12$, and $1$.
In our numerical experience, this may sometimes happen close to a bifurcation point, where the optimal network topology changes.
Indeed, the parameters of Figure\,\ref{fig:NonWorkingExample} lie almost exactly at a bifurcation where the truly optimal solution changes from four vertical pipes to a single tree.

\begin{figure}
\setlength{\unitlength}{\linewidth}
\begin{picture}(1.2,.52)
\put(0.01,.4){\small Urban} \put(.005,.38){\small Planning}
\put(.1,.29){\includegraphics[width=0.222\textwidth]{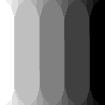}}
\put(.18,.27){\small $\varepsilon=0.015$} 
\put(.33,.29){\includegraphics[width=0.222\textwidth]{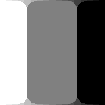}}
\put(.41,.27){\small $\varepsilon=0.2$} 
\put(.56,.29){\includegraphics[width=0.222\textwidth]{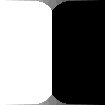}}
\put(.64,.27){\small $\varepsilon=0.8$} 
\put(.79,.29){\includegraphics[width=0.222\textwidth]{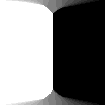}}
\put(.87,.27){\small $\varepsilon=1.5$} 
\put(.1,.29){\line(0,1){.222}}
\put(.1,.29){\line(1,0){.222}}
\put(.1,.512){\line(1,0){.222}}
\put(.33,.29){\line(0,1){.222}}
\put(.33,.29){\line(1,0){.222}}
\put(.33,.512){\line(1,0){.222}}
\put(.56,.29){\line(0,1){.222}}
\put(.56,.29){\line(1,0){.222}}
\put(.56,.512){\line(1,0){.222}}
\put(.79,.29){\line(0,1){.222}}
\put(.79,.29){\line(1,0){.222}}
\put(.79,.512){\line(1,0){.222}}
\put(-.006,.12){\small Branched} \put(-.007,.1){\small Transport}
\put(.1,.02){\includegraphics[width=0.222\textwidth]{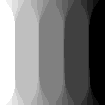}}
\put(.18,0){\small $\varepsilon=0.07$}  
\put(.33,.02){\includegraphics[width=0.222\textwidth]{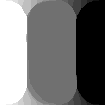}}
\put(.41,0){\small $\varepsilon=0.2$} 
\put(.56,.02){\includegraphics[width=0.222\textwidth]{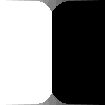}}
\put(.64,0){\small $\varepsilon=0.4$} 
\put(.79,.02){\includegraphics[width=0.222\textwidth]{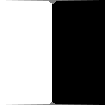}}
\put(.87,0){\small $\varepsilon=0.8$} 
\put(.1,.02){\line(0,1){.222}}
\put(.1,.02){\line(1,0){.222}}
\put(.1,.242){\line(1,0){.222}}
\put(.33,.02){\line(0,1){.222}}
\put(.33,02){\line(1,0){.222}}
\put(.33,.242){\line(1,0){.222}}
\put(.56,.02){\line(0,1){.222}}
\put(.56,.02){\line(1,0){.222}}
\put(.56,.242){\line(1,0){.222}}
\put(.79,.02){\line(0,1){.222}}
\put(.79,.02){\line(1,0){.222}}
\put(.79,.242){\line(1,0){.222}}
\put(.09,.01){\rule{2pt}{.5\unitlength}}
\put(-.005,.255){\rule{1.03\unitlength}{2pt}}
\end{picture}
\caption{Numerical optimization results for transport from 16 more or less evenly spaced point sources of same mass at the top to 16 evenly spaced point sinks (of same mass as well) at the bottom ($a=5$ in case of urban planning).
Instead of the optimal flux we show the corresponding optimal image $u$.}
\label{fig:Results16To16}
\end{figure}

Figure\,\ref{fig:Results16To16} shows numerical results obtained on a rectangular domain with 16 evenly spaced point sources of same mass at the top and as many point sinks (of same mass as well) at the bottom
(approximating a continuous line measure discretely).
Instead of the optimal flux we here show the corresponding optimal image $u$, whose jump set represents the network.
Depending on the chosen parameters, mass is preferentially collected in one or more tree-like networks before being separated again.
With decreasing parameter $\varepsilon$ (which in a way encodes the efficiency of transporting mass in bulk) the number of trees increases.

\begin{figure}
\setlength{\unitlength}{\linewidth}
\begin{picture}(1.2,.35)
\put(.085,.195){\includegraphics[width=0.145\textwidth]{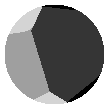}}
\put(.125,.174){\small $\varepsilon=0.4$} 
\put(.135,.336){$+$} \put(.077,.29){$-$} \put(.19,.325){$-$} \put(.17,.19){$-$} \put(.114,.194){$+$} \put(.216,.227){$+$}
\put(.245,.195){\includegraphics[width=0.145\textwidth]{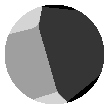}}
\put(.285,.174){\small $\varepsilon=1$} 
\put(.295,.336){$+$} \put(.237,.29){$-$} \put(.35,.325){$-$} \put(.33,.19){$-$} \put(.274,.194){$+$} \put(.376,.227){$+$}
\put(.405,.195){\includegraphics[width=0.145\textwidth]{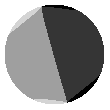}}
\put(.445,.174){\small $\varepsilon=2$} 
\put(.455,.336){$+$} \put(.397,.29){$-$} \put(.51,.325){$-$} \put(.49,.19){$-$} \put(.434,.194){$+$} \put(.536,.227){$+$}
\put(.565,.195){\includegraphics[width=0.145\textwidth]{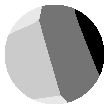}}
\put(.605,.174){\small $\varepsilon=0.2$} 
\put(.615,.336){$+$} \put(.557,.29){$-$} \put(.67,.325){\small \textbf{$+$}} \put(.65,.19){$-$} \put(.594,.194){$+$} \put(.696,.227){--}
\put(.725,.195){\includegraphics[width=0.145\textwidth]{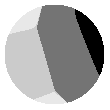}}
\put(.765,.174){\small $\varepsilon=0.7$} 
\put(.775,.336){$+$} \put(.717,.29){$-$} \put(.83,.325){\small \textbf{$+$}} \put(.81,.19){$-$} \put(.754,.194){$+$} \put(.856,.227){--}
\put(.885,.195){\includegraphics[width=0.145\textwidth]{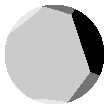}} 
\put(.925,.174){\small $\varepsilon=1.5$} 
\put(.935,.336){$+$} \put(.877,.29){$-$} \put(.99,.325){\small \textbf{$+$}} \put(.97,.19){$-$} \put(.914,.194){$+$} \put(1.016,.227){--}
\put(-.017,.275){\small Urban} \put(-.017,.255){\small Planning}
\put(-.012,.165){\rule{1.03\unitlength}{2pt}}
\put(.085,.005){\includegraphics[width=0.145\textwidth]{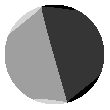}}
\put(.125,-.016){\small $\varepsilon=0.05$} 
\put(.135,.146){$+$} \put(.077,.1){$-$} \put(.19,.135){$-$} \put(.17,.0){$-$} \put(.114,.004){$+$} \put(.216,.037){$+$}
\put(.245,.005){\includegraphics[width=0.145\textwidth]{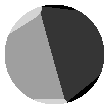}}
\put(.285,-.016){\small $\varepsilon=0.3$} 
\put(.295,.146){$+$} \put(.237,.1){$-$} \put(.35,.135){$-$} \put(.33,.0){$-$} \put(.274,.004){$+$} \put(.376,.037){$+$}
\put(.405,.005){\includegraphics[width=0.145\textwidth]{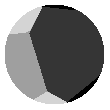}}
\put(.445,-.016){\small $\varepsilon=0.7$} 
\put(.455,.146){$+$} \put(.397,.1){$-$} \put(.51,.135){$-$} \put(.49,.0){$-$} \put(.434,.004){$+$} \put(.536,.037){$+$}
\put(.565,.005){\includegraphics[width=0.145\textwidth]{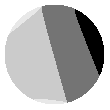}}
\put(.605,-.016){\small $\varepsilon=0.05$} 
\put(.615,.146){$+$} \put(.557,.1){$-$} \put(.67,.135){\small \textbf{$+$}} \put(.65,.0){$-$} \put(.594,.004){$+$} \put(.696,.037){--}
\put(.725,.005){\includegraphics[width=0.145\textwidth]{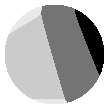}}
\put(.765,-.016){\small $\varepsilon=0.3$} 
\put(.775,.146){$+$} \put(.717,.1){$-$} \put(.83,.135){\small \textbf{$+$}} \put(.81,.0){$-$} \put(.754,.004){$+$} \put(.856,.037){--}
\put(.885,.005){\includegraphics[width=0.145\textwidth]{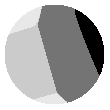}}
\put(.925,-.016){\small $\varepsilon=0.5$} 
\put(.935,.146){$+$} \put(.877,.1){$-$} \put(.99,.135){\small \textbf{$+$}} \put(.97,.0){$-$} \put(.914,.004){$+$} \put(1.016,.037){--}
\put(-.017,.085){\small Branched} \put(-.017,.065){\small Transport}
\put(.07,-.016){\rule{2pt}{.35\unitlength}}
\put(.553,-.016){\rule{2pt}{.35\unitlength}}
\end{picture}
\caption{Numerical optimization results for urban planning and branched transport with different parameters ($a=5$ in case of urban planning).
In the left column, the prescribed masses are $+\frac12,-\frac18,+\frac18,-\frac12,+\frac18,-\frac18$ (counterclockwise from top),
in the right column $+\frac12,-\frac18,+\frac18,-\frac12,-\frac12,+\frac12$.}
\label{fig:ResultsCircle1}
\end{figure}

Figure\,\ref{fig:ResultsCircle1} shows computational examples in which the underlying domain is non-rectangular and in which sources and sinks of different weights alternate on the domain boundary.
A more complex case is displayed in Figure\,\ref{fig:ResultsCircle2}, where we simulated 
the transport from a point source at the centre of the circular domain to 32 point sinks of equal mass on the boundary (approximating a uniform measure).
Such a geometry can be achieved by a trick using a periodic covering of the disk.
In detail, we chose the domain $\Omega=B_1(0)\setminus\{0\}$ and prescribed the sought image $u$ outside $\Omega$ taking values in $S^1$ rather than $\R$
(in between each two point sinks the image is prescribed to equal the midpoint between both sink positions on $S^1$).
Now $\R$ is interpreted as a covering of $S^1$ via the mapping $s\mapsto(\cos s,\sin s)$ so that the calculations can be performed as before
(similarly to the approach in \cite{CrSt13}).

\begin{figure}
\setlength{\unitlength}{\linewidth}
\begin{picture}(1.2,.47)
\put(0.03,.35){\small Urban} \put(.024,.33){\small Planning}
\put(.13,.25){\includegraphics[width=0.205\textwidth]{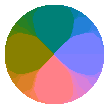}}
\put(.2,.235){\small $\varepsilon=0.05$} 
\put(.34,.25){\includegraphics[width=0.205\textwidth]{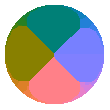}}
\put(.41,.235){\small $\varepsilon=0.2$} 
\put(.55,.25){\includegraphics[width=0.205\textwidth]{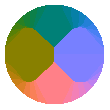}}
\put(.62,.235){\small $\varepsilon=0.4$} 
\put(.76,.25){\includegraphics[width=0.205\textwidth]{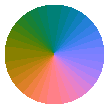}}
\put(.83,.235){\small $\varepsilon=1.5$} 
\put(.024,.12){\small Branched} \put(.022,.1){\small Transport}
\put(.13,.02){\includegraphics[width=0.205\textwidth]{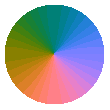}}
\put(.2,.005){\small $\varepsilon=0.01$}  
\put(.34,.02){\includegraphics[width=0.205\textwidth]{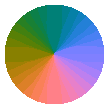}}
\put(.41,.005){\small $\varepsilon=0.05$} 
\put(.55,.02){\includegraphics[width=0.205\textwidth]{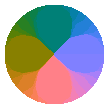}}
\put(.62,.005){\small $\varepsilon=0.25$} 
\put(.76,.02){\includegraphics[width=0.205\textwidth]{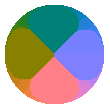}}
\put(.83,.005){\small $\varepsilon=0.4$} 
\put(.12,.005){\rule{2pt}{.44\unitlength}}
\put(.01,.225){\rule{.95\unitlength}{2pt}}
\end{picture}
\caption{Numerical optimization results for urban planning and branched transport for a single point source at the centre of the circular domain to 32 point sinks on the boundary ($a=5$ in case of urban planning).
The discontinuity set of the image corresponds to the optimal network.
For this geometry, the image $u$ takes values in $S^1$, which is here indicated by the periodic colour scale.}
\label{fig:ResultsCircle2}
\end{figure}

\section{Discussion}\label{sec:discussion}
The key idea to arrive at the convexification of network optimization in this work was the identification of network optimization with particular image inpainting problems, for which convexifications are known.

Let us stress that, even though we considered a simple rectangular or circular problem domain $\Omega$, the same approach can easily be applied for more general convex domains.

However, the formulation as image inpainting imposes two major restrictions on the network optimization problem.
First of all, such a formulation is only possible in two spatial dimensions because otherwise there cannot be any interpretation of rotated mass fluxes as image gradients.
Secondly, since image gradients are always curl-free, the corresponding mass fluxes must be divergence-free in the domain interior.
As a consequence, the source and sink distribution $\mu_+$ and $\mu_-$ must lie on the boundary of the domain.
Deviations from this situation may be possible (see Figure\,\ref{fig:ResultsCircle2}), but require special tricks.

The convexification derived in this work has only been shown to represent a lower bound to the original network optimization.
It still remains an open question when it is actually equivalent.
This question naturally decomposes into two separate issues:
First one has to investigate whether the network optimization problems and their Mumford--Shah versions are equivalent.
As explained in Section\,\ref{sec:inpainting}, this leads to analysing the regularity of the singular set of minimizers of Mumford--Shah-type functionals,
which is notoriously difficult and for which results are so far only partial.
For instance, the Mumford--Shah conjecture for the standard Mumford--Shah functional has been proved in two dimensions assuming that the singular set is connected by Bonnet in \cite{Bon96}.
The second issue concerns the tightness of the convexification via functional lifting.
For some lifted functionals there exist thresholding results showing that the solution of the original problem can be recovered from the convexified problem via thresholding (see for instance \cite{ChEsNi06}),
however, our numerical experiments indicate that this may sometimes not be the case.
Nevertheless, our proposed convexification turned out to be tight enough to derive the sharp lower bound in the energy scaling law for complicated networks,
and also in most numerical simulations it yielded the truly optimal network geometry.

\section{Acknowledgements}
This work was supported by the Deutsche Forschungsgemeinschaft (DFG), Cells-in-Motion Cluster of Excellence (EXC 1003-CiM), University of M\"unster, Germany.
B.W.'s research was supported by the Alfried Krupp Prize for Young University Teachers awarded by the Alfried Krupp von Bohlen und Halbach-Stiftung.

\bibliographystyle{alpha}
\bibliography{BrWi14}

\end{document}